\documentclass[11pt,oneside,reqno]{amsart}

\usepackage{algorithm, algorithmic, amsfonts, amssymb, amsthm, amsmath, array, bm, braket, caption, changepage, comment, dsfont, enumitem, esint, fullpage, graphicx, mathrsfs, mathtools, microtype, multicol, textcomp, tikz, tikz-cd, ulem}
\usepackage{hyperref}
\hypersetup{
	colorlinks=true,
	linkcolor=blue,
	citecolor=black,
	urlcolor=cyan}
\usepackage[foot]{amsaddr}

\usepackage{mmacells}

\usetikzlibrary{positioning}
\usetikzlibrary{shapes.geometric}

\usepackage{caption}
\usepackage{subcaption}
\usepackage{ulem}

\usepackage{mhequ}
\renewcommand{\aa}[1]{{#1}}

\newcommand{\om}[1]{{#1}}
\newcommand{\maybe}[1]{}

\usepackage{todonotes}
\newcommand{\vertiii}[1]{{\left\vert\kern-0.25ex\left\vert\kern-0.25ex\left\vert  #1 \right\vert\kern-0.25ex\right\vert\kern-0.25ex\right\vert}}

\usepackage{cleveref}
\crefname{alternative}{Alternative}{Alternatives}

\def\phi{\varphi}

\def\clj{C_{\ell j}}
\def\cjk{C_{jk}}
\def\ckl{C_{k \ell}}

\def\cjl{C_{\ell j}}
\def\ckj{C_{jk}}

\def\bkal{B_{k,b}^j}
\def\bjal{B_{j,b}^k}
\def\bkbl{B_{k,a}^j}
\def\bjbl{B_{j,a}^k}

\def\blak{B_{\ell,b}^j}
\def\bjak{B_{j,b}^\ell}
\def\blbk{B_{\ell,a}^j}
\def\bjbk{B_{j,a}^\ell}

\def\bkaj{B_{k,b}^\ell}
\def\blaj{B_{\ell,b}^k}
\def\bkbj{B_{k,a}^\ell}
\def\blbj{B_{\ell,a}^k}

\def\RR{\mathbb R}

\renewcommand{\epsilon}{\varepsilon}

\def\dphx{D_x\Phi_{h \tau}(x)}

\newtheoremstyle{}{}{}{}{}{}{}{ }{}

\numberwithin{equation}{section}

\theoremstyle{plain}
\newtheorem{theorem}{Theorem}[section]
\crefname{theorem}{Theorem}{Theorems}

\crefname{assumption}{Assumption}{Assumptions}
\newtheorem{cor}[theorem]{Corollary}
\crefname{cor}{Corollary}{Corollaries}
\newtheorem{definition}[theorem]{Definition}
\crefname{definition}{Definition}{Definitions}
\newtheorem{proposition}[theorem]{Proposition}
\Crefname{proposition}{Proposition}{Propositions}
\newtheorem{condition}{Condition}
\crefname{condition}{Condition}{Conditions}

\newtheorem{lemma}[theorem]{Lemma}
\crefname{lemma}{Lemma}{Lemmas}
\newtheorem{model}{Model}

\crefname{appendix}{Appendix}{Appendices}
\crefname{remark}{Remark}{Remarks}
\crefname{section}{Section}{Sections}

\newtheoremstyle{iremark}
  {}   
  {}   
  {\itshape}  
  {}       
  {\itshape}  
  {.}         
  {5pt plus 1pt minus 1pt} 
  {\thmname{#1}\thmnumber{ \itshape#2}\thmnote{ (#3)}} 
\theoremstyle{iremark}
\newtheorem{remark}[theorem]{Remark}
\newtheorem{alternative}{Alternative}

\DeclareMathOperator{\1}{\mathds{1}}

\DeclareMathOperator{\curl}{curl}
\DeclareMathOperator{\diverg}{div}
\DeclareMathOperator{\vol}{\textsc{v}}
\DeclareMathOperator{\Lie}{Lie}
\DeclareMathOperator{\ph}{\phantom{-}}
\DeclareMathOperator\supp{supp}

\newcommand{\pc}[1]{\left(#1\right)}

\newcommand{\pq}[1]{\left[#1\right]}

\newcommand{\eref}[1]{(\ref{#1})}

\title{Random splitting of fluid models:\\ Positive Lyapunov exponents}

\author{Andrea Agazzi$^{\star\sharp}$}
\author{Jonathan~C.~Mattingly$^{\star\diamond\dagger }$}
\author{Omar Melikechi$^\star$}
\address[$\star$]{Department of Mathematics, Duke  University, Durham, NC.}
\address[$\dagger$]{Department of Statistical Science, Duke
  University, Durham, NC.}
\address[$\diamond$]{Institute for Advanced Study, Princeton, NJ.}
\address[$\sharp$]{Department of Mathematics, University of Pisa, Pisa, IT.}

\begin{document}

\frenchspacing

\begin{abstract}
	In this paper, we apply the framework of \cite{Baxendale} to
        give sufficient conditions for the random splitting systems
        introduced in \cite{Agazzi} to have a positive top Lyapunov
        exponent. We verify these conditions for the randomly split
        conservative Lorenz-96 equations and randomly split Galerkin
        approximations of the 2D Euler equations on the torus. In
        doing so, we highlight particular structures in the equations
        such as shearing. Since a
        positive top Lyapunov exponent is an indicator of chaos which
        in turn is a feature of turbulence, our results show these
        randomly split fluid models have important characteristics of
        turbulent flow.
\end{abstract}

\maketitle


\section{Introduction}\label{sec:intro}


Operator splitting is a popular approach to numerically integrate the flow $\bar\Phi$ of a differential equation $\dot{x}=V(x)$ where $V$ is a vector field on $\mathbb{R}^D$. This method consists of expressing $V$ as a sum of simpler vector fields $\{V_j\}_{j=1}^n$ on $\mathbb{R}^D$ with flows $\varphi^{(j)}$ and approximating the flow $\bar\Phi$ by
\begin{align}\label{eq:splitting}
	\Phi(x,t) & \coloneqq \Phi_t(x)
	\coloneqq \varphi^{(n)}_{t_n}\circ\cdots\circ\varphi^{(1)}_{t_1}(x)
\end{align}
for a given choice of times $\{t_i\}_{i=1}^n$. In a simple version of such methods, $t_1 = \dots = t_n = h$ for $h \ll 1$ yields an approximation of the infinitesimal flow $\bar\Phi(x,h)$. To integrate $V$ on longer time intervals one concatenates \eqref{eq:splitting}, namely
$\Phi^m:\mathbb{R}^D\times\mathbb{R}^{mn}\to\mathbb{R}^D$ defined by
$\Phi^m(x,t)\coloneqq\Phi^m_t(x)\coloneqq\varphi^{(n)}_{t_{mn}}\circ\cdots\circ\varphi^{(1)}_{t_1}(x)$
with superscripts cycling in order from $1$ to $n$, again setting each
$t_i=h$. In \cite{Agazzi} we introduced a stochastic analog of
the above method called \textit{random splitting}. This stochastic
process is obtained by choosing the flow times \textit{randomly},
e.g., drawing $(\tau_j)_{j=1}^{\infty}$ independently and identically
distributed from a distribution $\rho$ with mean $1$ and
setting $t_j \equiv h \tau_j$, so that $\mathbb E(t_j) = h$. Similar to deterministic operator splitting, we proved in \cite{Agazzi} that under general conditions random splitting trajectories converge to the deterministic trajectories of $\bar\Phi$ almost surely as $h\to 0$. Thus random splitting is both a approximation method and a novel and useful way to introduce stochastic agitation.

In this paper we give sufficient conditions for a random splitting to have a positive \textit{top Lyapunov exponent}, \om{$\lambda_1(h)$, for all $h$ sufficiently small}. As detailed in \Cref{sec:exponents}, in our setting $\lambda_1(h)$ satisfies
\begin{align}\label{eq:top}
	\lambda_1(h) &= \lim_{m\to\infty}\tfrac{1}{m}\log\lVert D_x\Phi^m_{h\tau}(x)\rVert
\end{align}
for almost every $\tau$, where $\lVert\cdot\rVert$ is the operator norm. Our proof that $\lambda_1(h)>0$ involves adapting the framework of \cite{Baxendale}, which itself builds on ideas in \cite{Bougerol88b,Kifer86,FurstenbergKifer83,Guivarch80,FurstenbergKesten60},
to the random splitting setting. In particular, we are able
to reduce much of the work in proving positivity of the top Lyapunov
exponent for conservative systems to the verification of one
condition, the \textit{Lie bracket condition}, which roughly says the
random splitting dynamics can move in all directions in the vicinity
of any point at which it holds. A key step in our argument is proving that if the Lie bracket condition holds at a point and the vector fields of the random splitting are real analytic, then a positive power of the Markov transition kernel of the random splitting is strong Feller on a neighborhood of that point (\Cref{prop:split_feller}). We then show via results in \cite{Baxendale} that if the Lie bracket condition and some basic integrability criteria hold, and if $d\lambda_1(h)=\lambda_\Sigma(h)$ where $d$ is the dimension of the state space and $\lambda_\Sigma(h)$ is the sum of Lyapunov exponents, then either: (\textit{\Cref{alt:conformal}}) The dynamics are conformal with respect to some Riemannian structure on the state space or (\textit{\Cref{alt:subspaces}}) There exist proper subspaces of the tangent spaces to the state space that are invariant under the dynamics (Theorem \ref{thrm:main}). In particular, if $\lambda_\Sigma(h)=0$ then either $\lambda_1(h)>0$ or one of the two alternatives must hold. We then prove \Cref{alt:conformal} is ruled out by shearing and \Cref{alt:subspaces} is ruled out when the \om{Lie bracket condition holds for the \textit{lifted vector fields} $\widetilde V_j(x,\eta)=(V_j(x), DV_j(x)\eta)$ at a point $(x,\eta)$ in the tangent bundle of the state space. This is harder to verify in general than the Lie bracket condition for vector fields on the state space itself.}

Our primary motivation for developing the general results just described comes from the study of chaos and, ultimately, turbulence in fluid models. More specifically, we aim to explore the emergence of chaotic behavior in (the random splitting dynamics of) the following two models by proving their random splittings have a positive top Lyapunov exponent. Since a positive top Lyapunov exponent is an indicator of chaos which in turn is a precursor of turbulence, Theorems \ref{thrm:lorenz} and \ref{thrm:euler} below show these randomly split fluid models retain important characteristics of turbulent flow.

\begin{model}[\textit{Conservative Lorenz-96}]\label{mod:lorenz}
The {\normalfont conservative Lorenz-96 equations} describe the evolution of $x \in \mathbb{R}^n$  (here $D = n$) for $n\geq 4$ as
\begin{align}\label{eq:lorenz}
	\dot{x} & = V_{Lorenz}(x)
	\coloneqq \sum_{j=1}^n(x_{j+1}-x_{j-2})x_{j-1}e_j,
\end{align}
where $\{e_j\}_{j=1}^n$ is the
standard basis in $\mathbb{R}^n$ and indices are defined
periodically, i.e. $x_{-1}=x_{n-1}$, $x_0=x_n$, and $x_1=x_{n+1}$.
\end{model}

\begin{model}[\textit{2D Euler}]\label{mod:euler}
The {\normalfont $N$th-order Galerkin approximations of the Euler equations} on the 2-dimensional torus $\mathbb{T}$, approximating the behavior of the incompressible Euler equations
\begin{align}\label{e:eulerfull}
	\begin{cases}
		\partial_tu+(u\cdot\nabla)u = -\nabla p \\
		\diverg(u) \coloneqq \nabla\cdot u = 0,
	\end{cases}
\end{align}
where $u=(u_1,u_2):\mathbb{T}\times\mathbb{R}\to\mathbb{R}^2$ is the fluid velocity and $p:\mathbb{T}\times\mathbb{R}\to\mathbb{R}$ the fluid pressure. Concretely, the approximation is obtained by considering the Fourier transform $q(k,t)$ of the vorticity of $u$ and truncating the high-frequency modes with $|k|> N$. A precise description of this procedure and of the resulting ODE for $q$ is given in Section~\ref{sec:eulersplit}.
\end{model}

\noindent As detailed in \cite{Agazzi} and discussed in
\Cref{sec:models}, Models \ref{mod:lorenz} and \ref{mod:euler}
decompose into vector fields whose associated random splittings
preserve important physical properties of the original systems
such as energy conservation. In
particular, the split dynamics of both models are confined to submanifolds
of Euclidean space called \textit{$\mathcal{V}$-orbits}. Among these are certain ``typical'' $\mathcal{V}$-orbits called \textit{generic
  orbits} (\Cref{def:generic}), which are distinguished by having ``full measure'' in
  some sense. Applying the general results from above, we prove the
Lyapunov exponents of these random splittings on any generic orbit
$\mathcal{X}$ exist and are constant almost everywhere with
  respect to the volume form  on $\mathcal{X}$. Moreover, the
top Lyapunov exponent is positive. This is summarized in the following
theorems.

\begin{theorem}\label{thrm:lorenz}
\om{There exists $h_*>0$ such that} for every $n\geq 4$ the top Lyapunov exponent of the conservative Lorenz-96 random splitting on a generic
orbit is positive for all $h\in(0,h_*)$.
\end{theorem}

\begin{theorem}\label{thrm:euler}
\om{There exists $h_*>0$ such that} for every $N\geq 3$ the top Lyapunov exponent of the $N$th Galerkin approximated Euler random splitting on a generic orbit is positive for all $h\in(0,h_*)$.
\end{theorem}

\noindent The Lorenz and Euler random splittings are defined in Definitions \ref{def:lorenz_splitting} and \ref{def:euler_splitting}, respectively, and generic orbits are defined in \Cref{def:generic}. Proofs of Theorems \ref{thrm:lorenz} and \ref{thrm:euler} are in \Cref{sec:models2}.


\subsection*{Related work}


The random splitting defined in this work is an example of a
random dynamical system  \cite{Arnold98} or an integrated random
function \cite{DiaconisFreedman98}, and closely related to  piecewise deterministic
Markov process when the $\tau_j$ are chosen to be exponentially
distributed  and the ordering follows a Markov process \cite{Benaim}. We continue to describe the system as a
random splitting to highlight its underlying physical motivation,
i.e. its relation to $V$ and $\bar\Phi$ as above, and the
corresponding  physical structure that plays a central role in the
following analysis; however, many of the ideas can be applied to a
broader class of random dynamics.

It is natural to compare our work with \cite{BedrossianBlumenthalPunshonSmith2021,BedrossianPunshonSmith21} which consider the
related Lorenz-96 and Euler models with Brownian forcing and balancing dissipation. Here we work directly on the
conservative equations rather than removing dissipation and stochastic forcing through a limiting
procedure to approach the conservative dynamics. We expect this direct
approach to allow us to say more about the conservative dynamics.
Here the randomness is injected
through the random splitting and used mainly to make the dynamics
generic. We also expect this separation of the different roles of the
forcing to be useful. See \cite{Agazzi} for more discussion.

Our choice of how to inject randomness leads to a more ``elliptic''
dynamics in that noise directly affects most of the model's building blocks,
though the mechanism is arguably less disruptive than
elliptic additive Brownian forcing. While possibly more disruptive
than the ``minimally'' hypoelliptic\footnote{ The terms \textit{elliptic} and \textit{hypoelliptic} used here are analogous to the identical terms used in the partial differential equations literature.}  forcing considered in
\cite{BedrossianPunshonSmith21,hairermattingly06}, the random splitting
dynamics is likely more analytically tractable. For example, here we
are able to directly verify that the lifted dynamics on the projective
space satisfies the Lie bracket condition while in
\cite{BedrossianPunshonSmith21} (impressive) computer assisted algebra
is needed to verify the conditions. We are hopeful that this class of
models will lead to analyses in directions currently unfeasible for the more
ubiquitous models with Brownian forcing.

The regularity results and their proofs presented below are reminiscent of the probabilistic understanding, via Malliavin calculus, of H\"ormander's
classical results on hypoelliptic differential operators which require
Lie bracket conditions similar to those presented here
\cite{Hormander97}. Related results on the hypoellipticty of piecewise
deterministic Markov processes, though leading to slightly different
statements, can be found in \cite{Bakhtin,Benaim}.


\subsection*{Notation}


Define $\mathbb{R}_+\coloneqq (0,\infty)$, $\mathbb{R}_{\geq 0}\coloneqq [0,\infty)$, and, for $n\in\mathbb{N}$, $[n]\coloneqq \{1,\dots,n\}$. All other notation will be defined as it is introduced in the paper.


\subsection*{Organization of the paper}


In \Cref{sec:prelims} we introduce random splitting, the Lie
bracket condition, and Lyapunov exponents. In \Cref{sec:models} we define the conservative Lorenz-96 and Euler
equations and their random splittings, and apply to them the results
of \Cref{sec:prelims}. In \Cref{sec:main} we state our
main result, \Cref{thrm:main}, on Lyapunov exponents of general random
splittings. In \Cref{sec:feller} we prove the almost-sure
invertibility of a certain random matrix implies the strong Feller
property for a general class of transition kernels. We use this to
show the Lie bracket condition implies the strong Feller property for
analytic random splittings, which in turn proves our main general
result Theorem \ref{thrm:main}. In \Cref{sec:alternatives} we prove
shearing rules out \Cref{alt:conformal} and the Lie bracket condition
holding at a point in the tangent bundle of state space rules out
\Cref{alt:subspaces}. Finally, in \Cref{sec:models2} we use the
preceding results to prove \Cref{thrm:lorenz,thrm:euler}.


\section{Preliminaries}\label{sec:prelims}

We begin with our setting and preliminary results. \Cref{sec:splitting} introduces a general random splitting framework. \Cref{sec:lie} discusses the Lie bracket condition for random splittings and its relevant implications. \Cref{sec:exponents} introduces Lyapunov exponents of random splittings.


\subsection{Random splitting}\label{sec:splitting}


Let $\mathcal{V}\coloneqq\{V_j\}_{j=1}^n$ be a family of
complete\footnote{A vector field is \textit{complete} if its flow
  curve starting from any point exists for all time.},
analytic\footnote{Throughout this paper \textit{analytic} means \textit{real analytic}.} vector fields on $\mathbb{R}^D$ with flows $\varphi^{(j)}$ and define $\Phi:\mathbb{R}^D\times\mathbb{R}^n\to\mathbb{R}^D$ by \eref{eq:splitting}
\begin{equ}
	\Phi(x,t)  \coloneqq \Phi_t(x)
	\coloneqq \varphi^{(n)}_{t_n}\circ\cdots\circ\varphi^{(1)}_{t_1}(x).
\end{equ}
Similarly, define
$\Phi^m:\mathbb{R}^D\times\mathbb{R}^{mn}\to\mathbb{R}^D$ by
$\Phi^m(x,t)\coloneqq\Phi^m_t(x)\coloneqq\varphi^{(n)}_{t_{mn}}\circ\cdots\circ\varphi^{(1)}_{t_1}(x)$
with superscripts cycling in order from 1 to $n$. For each $x$ in
$\mathbb{R}^D$ the \textit{$\mathcal{V}$-orbit of $x$} is
\begin{align}\label{eq:orbit}
	\mathcal{X}(x) & \coloneqq \big\{ \Phi^m(x,t) : m\geq 0, t\in \mathbb{R}^{mn}\big\}.
\end{align}
This is the set of points that can be reached from $x$ in any finite
number of steps and over all times, positive and negative. When
  $x$ is arbitrary or understood, we denote $\mathcal{X}(x)$ by
  $\mathcal{X}$. Since the $V_j$ are complete, the
$\mathcal V$-orbits corresponding to different $x$
either agree or are disjoint. Hence we have an equivalence relation $x \sim y$ if and only if $\mathcal{X}(x) =
\mathcal X(y)$, and the $\mathcal{V}$-orbits $\{\mathcal
X(x):x\in\mathbb{R}^D\}$ partition $\mathbb{R}^D$. Furthermore, definition \eqref{eq:orbit} is equivalent to the
definition of $\mathcal{V}$-orbits in control theory \cite{jurdjevic,
  sussmann}. A classic result from geometric control theory ensures every $\mathcal{X}$ is an analytic submanifold of $\mathbb{R}^D$ \cite{jurdjevic}. In particular, each $\mathcal{X}$ has a Riemannian structure induced by the ambient Euclidean structure on $\mathbb{R}^D$ and an associated volume form $\vol$, sometimes called \textit{Lebesgue or Hausdorff measure on $\mathcal{X}$}, which will serve as our reference measure on $\mathcal{X}$.

To introduce randomness, fix $h>0$ and let
$\tau=(\tau_j)_{j=1}^\infty$ be mutually independent
$\mathbb{R}$-valued random variables with mean $1$ and common
distribution $\rho$. We assume $\rho$ is absolutely
continuous with respect to Lebesgue measure and, by a slight abuse of
notation, denote its density by $\rho$ as well,
i.e. $\rho(dt)=\rho(t)dt$. We also assume the support of $\rho$, $\supp(\rho)$,
contains an interval $(0,\epsilon)$ for some $\epsilon>0$ \aa{and that the tails of $\rho$ have an exponential (or faster) decay at infinity. The latter assumptions were made in \cite{Agazzi} to guarantee controllability of the dynamics and thus the uniqueness of an invariant measure on $\mathcal X$, and convergence of the random splitting process to its deterministic limit, respectively. These assumptions will again be used in the proof of \Cref{thrm:main}.} A
canonical choice for $\rho$ is the exponential distribution with mean 1;
the uniform distribution on $(0,2)$ also satisfies our
assumptions. \aa{We also denote the density of $h \tau_j$ by $\rho_h(t)$, i.e. $\rho_h(t):= h \rho(h t)$.} Setting $h\tau\coloneqq (h\tau_j)_{j=1}^\infty$, the \textit{random splitting associated to $\mathcal{V}$}, or just \textit{random splitting}, is the sequence $\{\Phi^m_{h\tau}\}_{m=0}^\infty$ where $\Phi^0_{h\tau}$ is the identity and
\begin{align}\label{eq:phi}
	\Phi^m_{h\tau} & \coloneqq \varphi^{(n)}_{h\tau_{mn}}\circ\cdots\circ\varphi^{(1)}_{h\tau_{(m-1)n+1}}(\Phi^{m-1}_{h\tau}).
\end{align}
That is, starting from the current step, the next step of the sequence
is obtained by flowing by each $V_j$ for the random time $h \tau_j$ in order from $j=1$ to
$n$.

Independence of the $\tau_j$ together with \eqref{eq:orbit} imply
$\{\Phi^m_{h\tau}\}$ is a Markov chain on $\mathcal{X}$ whenever the
sequence starts in $\mathcal{X}$.
Throughout this paper we denote the transition kernel of this Markov chain by $P_h$. It acts on functions $f:\mathcal{X} \rightarrow \RR$ by
$P_h f (x)\coloneqq \mathbb{E} f(\Phi_{h\tau}(x))$ and on measures $\mu$ on $\mathcal{X}$ by
\begin{align*}
	\mu P_h (A) &\coloneqq \int_{\mathcal{X}} P_h(x,A) \mu(dx),
\end{align*}
where $P_h(x,A) \coloneqq P_h \1_A(x)$ for measurable subsets $A$ of
$\mathcal{X}$.\footnote{The indicator function $\1_A(x)$ is 1 if $x \in A$ and 0 if $x \not \in A$} A measure $\mu$ is \textit{$P_h$-invariant} if $\mu P_h=\mu$.

\begin{remark}[Convergence as $h \rightarrow 0$]
	The term {\normalfont random splitting} alludes to the fact
        that, under quite general conditions, trajectories of the
        random dynamics $\{\Phi^m_{h\tau}\}_{m=0}^\infty$ converge
        almost-surely as $h\to 0$ to the trajectory of the
        deterministic ordinary differential equation
        $\dot{x}=V(x)={\textstyle\sum_{j=1}^n} V_j(x)$. See
          \cite[Section 4]{Agazzi}. We call $V$ the {\normalfont true vector field} and $V_j$ the {\normalfont splitting vector fields}.
\end{remark}


\subsection{The Lie bracket condition}\label{sec:lie}


The \textit{Lie bracket condition} is a key condition on a family of vector
fields $\mathcal V$ that guarantees sufficient nondegeneracy of its dynamics. Specifically, fix a $d$-dimensional
$\mathcal V$-orbit $\mathcal{X}$ of a general random splitting as in Section \ref{sec:splitting}. Let
$\mathfrak{X}(\mathcal{X})$ denote the Lie algebra of smooth vector
fields on $\mathcal{X}$ and define $\Lie(\mathcal{V})$ to be the
smallest subalgebra of $\mathfrak{X}(\mathcal{X})$ containing
$\mathcal{V}$. For each $x\in\mathcal{X}$ the collection
$\Lie_x(\mathcal{V})\coloneqq\{V(x):V\in\Lie(\mathcal{V})\}$ is a
subspace of $T_x\mathcal{X}$.

\begin{definition}
  The {\normalfont Lie bracket condition} holds at $x\in\mathcal{X}$ if $\Lie_x(\mathcal{V})=T_x\mathcal{X}$.
\end{definition}

The Lie bracket condition is called the \textit{weak bracket
	condition} in \cite{Benaim} and \textit{Condition B} in
\cite{Bakhtin}. Both papers also consider a \textit{strong bracket
	condition} (\textit{Condition A}) which is used for results about
continuous time Markov processes and is therefore not needed here. The
Lie bracket condition has the following important consequence.

\begin{theorem}\label{thrm:submersion}
If the Lie bracket condition holds at $x_*$ then for every neighborhood $U$ of $x_*$ and every $T>0$ there exists $x\in U$, $\kappa\in\mathbb{N}$, and $t\in\mathbb{R}^{\kappa n}_+$ such that $\sum_{k=1}^{\kappa n} t_k\leq T$ and $t\mapsto \Phi^\kappa_t(x_*)=x$ is a submersion at $t$, i.e. $D_t\Phi^\kappa_t(x_*):T_t\mathbb{R}^{\kappa n}\to T_x\mathcal{X}$ is surjective.
\end{theorem}

\begin{proof}
\om{Fix $T>0$ and a neighborhood $U$ of $x_*$. In our notation \cite[Theorem 5]{Bakhtin} says that for all $i,j\in [n]$ there exists $\kappa'>d$, a sequence of indices ${\bf i}=(i_1,\dots,i_{\kappa'+1})\in [n]^{\kappa'+1}$ with $i_1=i$ and $i_{\kappa'+1}=j$, and $t\in\mathbb{R}^{\kappa'+1}_+$ with $\sum_{k=1}^{\kappa'+1} t_k < T/2$ such that $F_{\bf i}:\mathbb{R}^{\kappa'+1}_+\to\mathcal{X}$ defined by $F_{\bf i}(t)=\varphi^{(i_{\kappa'+1})}_{t_{\kappa'+1}}\circ\cdots\circ\varphi^{(i_1)}_{t_1}(x_*)$ is a submersion at $t$. To adapt this to our setting in which the vector fields are taken in a specific order, first observe we can extend $F_{\bf i}$ to $\Phi^\kappa$ for some $\kappa\geq\kappa'$ by flowing for time zero in the appropriate slots. For example, if $n=2$ and $F_{\bf i}(t_1,t_2)=\varphi^{(1)}_{t_2}\circ\varphi^{(2)}_{t_1}(x_*)$, then $F_{\bf i}$ is equal to $\Phi^2_t(x_*)=\varphi^{(2)}_0\circ\varphi^{(1)}_{t_2}\circ\varphi^{(2)}_{t_1}\circ\varphi^{(1)}_0(x_*)$ where $t=(0,t_1,t_2,0)$. So if $F_{\bf i}$ is surjective at $(t_1,t_2)$, then $t\mapsto\Phi^2_t(x_*)$ is a submersion at $(0,t_1,t_2,0)$. In the general case we therefore have that there exists $\kappa$ and $t\in[0,\infty)^{\kappa n}$ such that $\sum_{k=1}^{\kappa n} t_k\leq T/2$ and $t\mapsto\Phi^\kappa_t(x_*)$ is a submersion at $t$. We then extend to strictly positive times by using the fact that surjectivity is an open condition. This can be seen, for example, by noting that $\lvert \det D_t\Phi^\kappa_t(x_*)D_t\Phi^\kappa_t(x_*)^*\rvert > 0$ at $t$, where $A^*$ denotes the transpose of a matrix $A$. By continuity there must be an open neighborhood $\Delta$ of $t$ such that $\lvert \det D_s\Phi^\kappa_t(x_*)D_t\Phi^\kappa_t(x_*)^*\rvert > 0$ for all $s\in U$. In particular, there must exist $s$ with strictly positive entries such that $\sum_{k=1}^{\kappa n} s_k < T$ and $t\mapsto D_t\Phi^\kappa_t(x_*)$ is a submersion at $s$.}
\end{proof}

\noindent Intuitively Theorem \ref{thrm:submersion}
says that if the Lie bracket condition holds at $x_*$ then, as a
consequence of surjectivity, the random splitting can move in any
infinitesimal direction from $x_*$ in arbitrarily small positive
times. In Section \ref{sec:feller}, we combine this fact with
analyticity of the $V_j$ to prove a certain almost-sure surjectivity of
the dynamics near $x_*$. This, in turn, allows us to show that for $\kappa$
as in Theorem \ref{thrm:submersion} the transition kernel $P^\kappa_h$ is
strong Feller\footnote{See \Cref{sec:feller} for a definition of strong Feller.} on a neighborhood of $x_*$ for every $h>0$.


\subsection{Lyapunov exponents}\label{sec:exponents}


Let $\{\Phi^m_{h\tau}\}$ be a random splitting on a $d$-dimensional orbit
$\mathcal{X}$ with transition kernel $P_h$ and ergodic
$P_h$-invariant measure $\mu$. Set $\log^+(a)\coloneqq\max\{\log a,0\}$ for $a>0$. If
\begin{align}\label{eq:int_cond}
	\mathbb{E}\int_\mathcal{X}\left(\log^+\lVert D_x\Phi_{h\tau}(x)\rVert + \log^+\lVert D_x\Phi_{h\tau}(x)^{-1}\rVert\right)\mu(dx) & < \infty,
\end{align}
then the multiplicative ergodic theorem guarantees
the existence of $d$ numbers $\lambda_1(h)\geq\cdots\geq\lambda_d(h)$, called the
\textit{Lyapunov exponents} of $\{\Phi^m_{h\tau}\}$, such that for
$\mu$-almost every $x\in\mathcal{X}$ and every $\eta\in T_x\mathcal{X}$,
\begin{align*}
	\lim_{m\to\infty}\tfrac{1}{m}\log \lVert D_x\Phi^m_{h\tau}(x)\eta\rVert & = \lambda_k(h)
\end{align*}
for some $k$ and almost every $\tau$. Here $D_x$ is the derivative in
$x$ and $\lVert\cdot\rVert$ is the norm on tangent spaces of
$\mathcal{X}$ induced by the Euclidean norm on
$\mathbb{R}^D$.
Moreover, the \textit{top Lyapunov exponent}, $\lambda_1(h)$, satisfies \eref{eq:top} and the \textit{sum of the Lyapunov exponents}, $\lambda_\Sigma(h)\coloneqq\lambda_1(h)+\cdots+\lambda_d(h)$, satisfies
\begin{align}\label{eq:sum}
	\lambda_\Sigma(h) &\coloneqq \lim_{m\to\infty}\tfrac{1}{m}\log\big\lvert \det\big(D_x\Phi^m_{h\tau}(x)\big)\big\rvert.
\end{align}

\noindent The top Lyapunov exponent captures the largest rate of separation
between trajectories $\{\Phi^m_{h\tau}(x)\}$ and
$\{\Phi^m_{h\tau}(y)\}$ for points $y$ infinitesimally close to $x$. A
positive value corresponds to exponential growth and is a hallmark of
chaos. The sum of the Lyapunov exponents captures the overall behavior
of volumes under the dynamics with $\lambda_\Sigma(h)>0$,
$\lambda_\Sigma(h)=0$, and $\lambda_\Sigma(h)<0$ indicative of expanding,
conservative, and contracting dynamics, respectively. As will be the case in
our examples, it is possible (and in some sense typical) to have
a system which conserves volumes
($\lambda_\Sigma(h)=0$) while having some expanding directions ($\lambda_1(h)>0$).

\aa{We now introduce \Cref{cond:bound}, which is used} throughout this paper. In particular, \Cref{lem:bound} says it implies \eqref{eq:int_cond} for all $h>0$ and hence that Lyapunov exponents exist and satisfy \eqref{eq:top} and \eqref{eq:sum}. \aa{The proof uses the exponential tails of $\rho$, which compensate for the growth in $t$ of the Jacobians of $\Phi_t(x)$.} Since the $V_j$ are smooth, \Cref{cond:bound} holds whenever $\mathcal{X}$ is bounded as a subset of $\mathbb{R}^D$.

\begin{condition}\label{cond:bound}
\om{The splitting vector fields $V_j$ and their first and second derivatives are uniformly bounded on $\mathcal{X}$. That is, there exists a constant $C$ such that
\begin{align}\label{eq:bound}
	\sup\{\lVert V_j\rVert, \lVert DV_j(x)\rVert, \lVert D^2V_j(x)\rVert : x\in\mathcal{X}, 1\leq j\leq n\} & \leq C
		< \infty.
\end{align}}
\end{condition}

\begin{lemma}\label{lem:bound}
If \Cref{cond:bound} holds then \eqref{eq:int_cond} holds for all $h>0$. \om{Furthermore, for any nonnegative integers $m_1$, $m_2$, and $m_3$ there exists $h_*>0$ (depending on $m_1$, $m_2$, and $m_3$) such that
\begin{align*}
	\sup\left\{\mathbb{E}\lVert D_x\Phi_{h\tau}(x)^{-1}\rVert^{m_1}\lVert D_x\Phi_{h\tau}(x)\rVert^{m_2}\lVert \partial_h D_x\Phi_{h\tau}(x)\rVert^{m_3} : x\in\mathcal{X}, h\in [0,h_*)\right\} &< \infty.
\end{align*}}
\end{lemma}

\begin{proof}
Fix $x\in\mathcal{X}$ and $t \in \mathbb{R}^n_+$ and let $C$ be as in \eqref{eq:bound}. \aa{By the chain rule,
\begin{align*}
	D_x\Phi_t(x) & = D_x\varphi^{(n)}_{t_n}\big(\Phi_t^{(1, n-1)}(x)\big)\cdots D_x\varphi^{(2)}_{t_2}\big(\Phi_t^{(1, 1)}(x)\big)D_x\varphi^{(1)}_{t_1}\big(x\big)
\end{align*}
where for $1 \leq a \leq b \leq n$ we define
\begin{align*}
	\Phi_t^{(a,b)}(x)\coloneqq \varphi^{(b)}_{t_b}\circ\cdots\circ\varphi^{(a)}_{t_a}(x)\,.
\end{align*}
}
	Therefore
\begin{align}\label{eq:submult}
	\lVert D_x\Phi_t(x)\rVert & \leq \prod_{j=1}^n \big\lVert D_x\varphi^{(j)}_{t_k}\big(x^{(j-1)}\big)\big\rVert.
\end{align}
Fix $j$ and set $\varphi=\varphi^{(j)}$. By Cauchy-Schwarz, for any unit vector $\eta \in T_x\mathcal{X}$,
\begin{align*}
	\partial_s\lVert D_x\varphi_s(x)\eta\rVert^2 & \leq 2\big\lVert DV_k(\varphi_s(x))\rVert\lVert D_x\varphi_s(x)\eta\rVert^2
		\leq 2C\lVert D_x\varphi_s(x)\eta\rVert^2.
\end{align*}
So by Gr\"onwall's inequality, $\lVert D_x\varphi_s(x)\eta\rVert\leq\exp(Cs)$. Since $j$, $x$, and $\eta$ were arbitrary, \eqref{eq:submult} gives
\begin{align}\label{eq:exp_bound}
	\sup_{x\in\mathcal{X}}\lVert D_x\Phi_{h\tau}(x)\rVert & \leq \exp\left(Ch\sum_{j=1}^n\tau_j\right),
\end{align}
which also holds for $\lVert D_x\Phi_{h\tau}(x)^{-1}\rVert$ since \om{$\Phi_t(x)^{-1}=\varphi^{(1)}_{-t_1}\circ\cdots\circ\varphi^{(n)}_{-t_n}(x)$ and neither \eqref{eq:submult} nor the Gr\"onwall estimate depend on the order of the vector fields or the signs of the times}. The integrability condition \eqref{eq:int_cond} follows immediately since $Ch\sum\tau_j$ is integrable against the exponential tails of $\rho$ for all $h>0$. \om{For the ``furthermore" part of the lemma, Cauchy-Schwarz inequality gives
\begin{equation}\label{eq:cauchy}
\begin{aligned}
	&\mathbb{E}\lVert D_x\Phi_{h\tau}(x)^{-1}\rVert^{m_1}\lVert D_x\Phi_{h\tau}(x)\rVert^{m_2}\lVert \partial_h D_x\Phi_{h\tau}(x)\rVert^{m_3} \\
		&\quad\leq \left(\mathbb{E}\lVert D_x\Phi_{h\tau}(x)^{-1}\rVert^{2 m_1}\lVert D_x\Phi_{h\tau}(x)\rVert^{2m_2}\right)^{\frac{1}{2}}\left(\mathbb{E}\lVert \partial_h D_x\Phi_{h\tau}(x)\rVert^{2m_3}\right)^{\frac{1}{2}}.
\end{aligned}
\end{equation}
By \eqref{eq:exp_bound} and the exponential tails of $\rho$, there exists $h_*>0$ such that
\begin{align}\label{eq:first_term}
	\sup_{x\in\mathcal{X}, h\in[0,h_*)}\left(\mathbb{E}\lVert D_x\Phi_{h\tau}(x)^{-1}\rVert^{2 m_1}\lVert D_x\Phi_{h\tau}(x)\rVert^{2m_2}\right)^{\frac{1}{2}} &< \infty.
\end{align}
To control $\left(\mathbb{E}\lVert \partial_h D_x\Phi_{h\tau}(x)\rVert^{2m_3}\right)^{\frac{1}{2}}$, for $i<j$ define $\Phi^{(i,j)}\coloneqq\varphi^{(j)}\circ\cdots\circ\varphi^{(i)}$ and observe
\begin{align*}
	\sup_{x \in  \mathcal X}\|\partial_h \dphx\| &= \sup_{x \in  \mathcal X}\|D_x\partial_h \Phi_{h \tau}(x)\| = \sup_{x \in  \mathcal X} \|D_x \sum_{j=1}^n \tau_j \partial_{h \tau_j} \Phi_{h \tau}(x)\|\\
		&= \sup_{x \in  \mathcal X} \|D_x \sum_{j=1}^n \tau_j  D_x\Phi_{h \tau}^{(j+1,n)}(\Phi_{h \tau}^{(1,j)}(x)) V_j(\Phi_{h \tau}^{(1,j)}(x))\|\\
		& \leq \sum_{j} \tau_j \sup_{x \in  \mathcal X} \|D_x \pc{D_x\Phi_{h \tau}^{(j+1,n)}(\Phi_{h \tau}^{(1,j)}(x)) V_j(\Phi_{h \tau}^{(1,j)}(x))} \|\\
		& \leq \sum_{j} \tau_j \sup_{x \in  \mathcal X}\Big( \| D_x^2 \Phi_{h \tau}^{(j+1,n)}(\Phi_{h \tau}^{(1,j)}(x)) [D_x\Phi_{h \tau}^{(1,j)}(x),V_j(\Phi_{h \tau}^{(j)}(x))] \|\\
		& \qquad \qquad \qquad \qquad+ \| D_x\Phi_{h \tau}^{(j+1,n)}(\Phi_{h \tau}^{(1,j)}(x)) D_xV_j(\Phi_{h \tau}^{(1,j)}(x))D_x\Phi_{h \tau}^{(1,j)}(x) \| \Big)\\
		& \leq \sum_{j} \tau_j \sup_{x \in \mathcal X} \Big(\|D_x^2 \Phi_{h \tau}^{(j+1,n)}(x)\|\|D_x\Phi_{h \tau}^{(1,j)}(x)\|\|V_j(x)\| \\
		&\qquad\qquad\qquad\qquad+\|D_x\Phi_{h \tau}^{(j+1,n)}(x)\|\| D V_j(x)\|\|D_x\Phi_{h \tau}^{(1,j)}(x)\|\Big).
\end{align*}
Since $\|V_j(x)\|, \| D V_j(x)\|$ are uniformly bounded by assumption and $\lVert D_x\Phi_{h\tau}(x)^{-1}\rVert$ and $\lVert D_x\Phi_{h\tau}(x)\rVert$ are controlled by the preceding argument, it remains only to control $\lVert D_x^2\Phi_{h\tau}(x)\rVert$. To that end, write
\begin{align*}
	\sup_{x \in \mathcal X}\|D_x^2 \Phi_{h \tau}^{(a,b)}(x)\| \leq \sum_{j=a}^b \sup_{x \in \mathcal X}\|D_x^2 \phi_{h \tau_j}^{(j)}(x)\| \prod_{k = a}^b \max\{1, \sup_{x \in \mathcal X} \|D_x \phi_{h \tau_k}^{(k)}(x)\|^2\}.
\end{align*}
Since all the terms in the product grow at most exponentially in $h\tau$ by the bounds above, we only have to show that the same holds for $\|D_x^2 \phi_{h \tau_j}^{(j)}(x)\|$. This is done as above with Gr\"onwall's inequality:
\begin{align*}
	\partial_s \|D_x^2\phi_{s}(x)\| & \leq  \|D_x(D V_j(\phi_s(x)) \cdot D_x\phi_s(x))\|\\
		&\leq \sup_{x' \in \mathcal X}\|D_x^2 V(x')\| \|D_x\phi_{s}(x)\|^{2} + \sup_{x' \in \mathcal X}\|D_x V(x')\| \|D_x^2\phi_{s}(x)\|\\
		 & \leq C e^{C s} + C \|D_x^2\phi_{s}(x)\|
\end{align*}
yielding $\|D_x^2\phi_s(x)\| \leq C' e^{C' s}$ as required. Thus, shrinking $h_*$ if necessary, we have
\begin{align*}
	\sup_{x\in\mathcal{X}, h\in[0,h_*)}\left(\mathbb{E}\lVert \partial_h D_x\Phi_{h\tau}(x)\rVert^{2m_3}\right)^{\frac{1}{2}} &< \infty,
\end{align*}
which together with \eqref{eq:cauchy} and \eqref{eq:first_term} yields the desired result.}
\end{proof}



\section{Motivating models: Conservative Lorenz-96 and 2d Euler}\label{sec:models}

Before giving
  our main results for general random splittings in Sections
  \ref{sec:main} and \ref{sec:feller}, in this section we motivate
  the framework of \Cref{sec:prelims} by applying it to the
  conservative Lorenz-96 and 2-dimensional Euler equations introduced in \Cref{sec:intro}. Specifically, in \Cref{sec:model_split} we
define random splittings of the conservative Lorenz-96 and
Galerkin approximated 2-dimensional Euler models. In
\Cref{sec:model_results} we discuss important features of these
models: their conservation laws, generic orbits, the existence of a
unique invariant measure on each generic orbit, and existence
of Lyapunov exponents.

\subsection{Random splittings of Lorenz and Euler}\label{sec:model_split}

\subsubsection*{Conservative Lorenz-96}\label{sec:lorenzsplit}
Recall the \textit{conservative Lorenz-96 equations} introduced in \eref{eq:lorenz},
\begin{equ}
	\dot{x} = V_{Lorenz}(x)
		\coloneqq \sum_{j=1}^n(x_{j+1}-x_{j-2})x_{j-1}e_j,
\end{equ}
where $n\geq 4$, $x\in\mathbb{R}^n$, $\{e_j\}_{j=1}^n$ is the standard basis in $\mathbb{R}^n$, and indices are defined periodically, i.e. $x_{-1}=x_{n-1}$, $x_0=x_n$ and $x_1=x_{n+1}$. The original vector field $V_{Lorenz}$ splits as
\begin{align}\label{eq:split_lorenz}
	V_{Lorenz} &= \sum_{j=1}^n V_j
		\quad\text{where}\quad
		V_j(x) \coloneqq x_{j-1}(x_{j+1}e_j-x_je_{j+1}).
\end{align}
Each vector field $V_j$ generates a rotation in the $(x_j,x_{j+1})$-plane with angular velocity $x_{j-1}$.

\begin{definition}[Lorenz splitting]\label{def:lorenz_splitting}
A {\normalfont random splitting of conservative Lorenz-96}, or just {\normalfont Lorenz splitting}, is any random splitting associated to the family $\mathcal{V}_{Lorenz}\coloneqq\{V_j\}_{j=1}^n$ defined in \eqref{eq:split_lorenz}.
\end{definition}

\subsubsection*{2D Euler} \label{sec:eulersplit}
For any $N\geq 3$ the \textit{$N$th Galerkin approximation of the Euler equations on the 2-dimensional torus}, $\mathbb{T}$, is given by
\begin{align}\label{eq:galerkin}
	\dot{q} & = V_{Euler}(q)
	\coloneqq -\left(\sum_{j+k+\ell=0} C_{k\ell}\bar{q}_k\bar{q}_\ell\right)e_j,
\end{align}
where $j$, $k$, $\ell$ range over $\mathbb{Z}^2_N\coloneqq \{j\in\mathbb{Z}^2:\max\{\lvert j_1\rvert, \lvert j_2\rvert\}\leq N\}$, $\{e_j : j\in\mathbb{Z}^2\}$ is the orthonormal basis $e_j(x)\coloneqq (2\pi)^{-1}\exp(ix\cdot j)$ of $L^2(\mathbb{T},\mathbb{C})$, $q\in\mathbb{C}^n$ with $n=4N(N+1)$, and
\begin{align}\label{eq:constants}
	C_{k\ell} & = \frac{\langle k,\ell^\perp\rangle}{4\pi}\left(\frac{1}{\lvert k\rvert^2} - \frac{1}{\lvert \ell\rvert^2}\right)
\end{align}
with $\ell^\perp\coloneqq (\ell_2,-\ell_1)$. \aa{Note that if $j+k+\ell = 0$ and $j,k,\ell$ are not all distinct then $C_{jk} = C_{k\ell} = C_{\ell j} = 0$, so the vector fields in this case are identically $0$ on all coordinates.} A detailed derivation of \eqref{eq:galerkin} is given in \cite{Agazzi}. The variable $q$ in \eqref{eq:galerkin} is the restriction of the Fourier transform of the vorticity, $\curl(u(x,t))$ from \eref{e:eulerfull} to the Fourier modes whose indices are in $\mathbb{Z}^2_N$. We will often regard $q$ as an element of $\mathbb{R}^{2n}$ and write $q=(a,b)$ where $a$ and $b$ are the real and imaginary parts of $q$, respectively.

The true vector field $V_{Euler}$ from \eqref{eq:galerkin} split further into real vector fields as
\begin{align}\label{split_euler}
	V_{Euler} & = \sum_{j+k+\ell=0} V_{a_ja_ka_\ell} + V_{a_jb_kb_\ell} + V_{b_ja_kb_\ell} + V_{b_jb_ka_\ell}
\end{align}
with indices ranging over $\mathbb{Z}^2_N$ and$V_{a_ja_ka_\ell}, V_{a_jb_kb_\ell}, V_{b_ja_kb_\ell}$, and $V_{b_jb_ka_\ell}$ given by
\begin{align}\label{eq:euler_fields}
	\begin{cases}
		\dot{a}_j = -C_{k\ell}a_ka_\ell \\
		\dot{a}_k = -C_{j\ell}a_ja_\ell \\
		\dot{a}_\ell = -C_{jk}a_ja_k
	\end{cases}
	\begin{cases}
		\dot{a}_j = C_{k\ell}b_kb_\ell \\
		\dot{b}_k = C_{j\ell}a_jb_\ell \\
		\dot{b}_\ell = C_{jk}a_jb_k
	\end{cases}
	\begin{cases}
		\dot{b}_j = C_{k\ell}a_kb_\ell \\
		\dot{a}_k = C_{j\ell}b_jb_\ell \\
		\dot{b}_\ell = C_{jk}b_ja_k
	\end{cases}
	\begin{cases}
		\dot{b}_j = C_{k\ell}b_ka_\ell \\
		\dot{b}_k = C_{j\ell}b_ja_\ell \\
		\dot{a}_\ell = C_{jk}b_jb_k,
	\end{cases}
\end{align}
\aa{respectively, with all other components equal to $0$. Here,} $a_j$ and $b_j$ are the real and imaginary parts of the $j$th Fourier mode $q_j$. The slight difference in the signs between \eref{eq:euler_fields} and \cite[Equation (6.7)]{Agazzi} is because here triples $(j,k,\ell)$ satisfy $j + k + \ell = 0$ rather than $j + k - \ell = 0$ as in that reference.

\begin{definition}[Euler splitting]\label{def:euler_splitting}
A {\normalfont random splitting of the $N$th Galerkin approximation of 2-dimensional Euler}, or just {\normalfont Euler splitting}, is any random splitting associated to the family $\mathcal{V}_{Euler}\coloneqq\{V_{a_ja_ka_\ell}, V_{a_jb_kb_\ell}, V_{b_ja_kb_\ell}, V_{b_jb_ka_\ell} : j,k,\ell\in\mathbb{Z}^2_N, j+k+\ell=0\}$ defined in \eqref{eq:euler_fields}.
\end{definition}

\begin{remark}
The splitting vector fields in Definitions \ref{def:lorenz_splitting} and \ref{def:euler_splitting} are unambiguous. However, we have said ``\textit{a random splitting of...}" rather than ``\textit{the random splitting of...}" because one still has to choose the common distribution of the random times \aa{and the order of the vector fields. This is especially important in the Euler case where there are multiple ``natural'' orderings for the interacting triples. We note that while the specific Lyapunov exponents of the system may depend on choice of ordering, our proof is unaffected by such a choice as long as it is preserved across cycles, reflecting the fact that the positivity of the top Lyapunov exponent is independent of ordering.}
\end{remark}

\begin{remark}[Shear isolating splitting]\label{rem:shearSplitting}
  An important feature of the above splittings is that they both
  maintain (and highlight) the shearing which is expected to be
  important in the original dynamics for producing many of the systems'
  dynamical features such as positive Lyapunov exponents. To see the
  shearing in \eqref{eq:split_lorenz} notice that the speed of
  rotation of $(x_j,x_{j+1})$ varies depending on the value of
  $x_{j-1}$. The same mechanism of
  nearby points rotating at different rates also produces shear in
  \eqref{eq:euler_fields}.
\end{remark}

\begin{remark}[Euler tops]
  As an interesting aside, we note that each of the four sets of equations given in
  \eqref{eq:euler_fields} are those of a classical
  Euler spinning top. Hence, \eqref{split_euler} gives the
  decomposition of the Euler equations for fluid motion into a
  collection of interacting Euler tops. See
  \cite{Gluhovsky_1993,Glukhovsky_1982} for investigations of related
  deterministic decompositions used to model turbulence.
\end{remark}


\subsection{Preliminary results for splittings of Lorenz and Euler}\label{sec:model_results}


The conservative Lorenz-96 equations \eqref{eq:lorenz} and Galerkin
approximated Euler equations \eqref{eq:galerkin} conserve Euclidean
norm on \aa{$\mathbb{R}^{2n}$ (for the latter, this is a consequence of $C_{jk} + C_{k\ell} + C_{\ell j} = 0$ for $j+k+\ell = 0$)}. So too do all the splitting vector fields in
$\mathcal{V}_{Lorenz}$ and $\mathcal{V}_{Euler}$. Thus in both random
splitting models the dynamics are confined to spheres centered at the
origin with the same radius as the initial data. Not all
splittings are guaranteed to have this property; we chose the
splittings $\mathcal{V}_{Lorenz}$ and $\mathcal{V}_{Euler}$ in part
because they conserve the Euclidean norm, as do the original dynamics.

By the above discussion the $\mathcal{V}_{Lorenz}$-orbit of any point $x$ satisfies
\begin{align*}
	\mathcal{X}(x) &\subseteq \big\{y\in\mathbb{R}^n : \lVert y\rVert = \lVert x\rVert\big\}
		\eqqcolon \mathcal{M}_{Lorenz}(x).
\end{align*}
While a similar statement holds for Euler, more can be said in this case. Indeed, in addition to the Euclidean norm which corresponds to enstrophy in Galerkin approximated Euler, \eref{e:eulerfull} conserves a second quantity corresponding to the energy of the original fluid system. In the variable $q=(a,b)$, the energy is given by $E(a,b)=\sum_{k \in \mathbb{Z}^2} \tfrac{|a_k|^2 + |b_k|^2}{|k|^2}$. Direct calculation shows $E(a,b)$ is also conserved by every splitting vector field in $\mathcal{V}_{Euler}$ \cite{Agazzi}. Thus the $\mathcal{V}_{Euler}$-orbit of any $q=(a,b)$ satisfies
\begin{align*}
	\mathcal{X}(q) &\subseteq \big\{(\alpha,\beta)\in\mathbb{R}^{2n} : \lVert(\alpha,\beta)\rvert=\lVert(a,b)\rVert, E(\alpha,\beta)=E(a,b)\big\}
		\eqqcolon \mathcal{M}_{Euler}(q),
\end{align*}
where we recall $n=4N(N+1)$ for the $N$th Galerkin approximation of Euler. Note the collections of manifolds $\{\mathcal{M}_{Lorenz}(x):x\in\mathbb{R}^n\}$ and $\{\mathcal{M}_{Euler}(q):q\in\mathbb{R}^{2n}\}$ foliate their respective ambient Euclidean spaces. Any $\mathcal{M}$ from either of these foliations further decomposes into $\mathcal V$-orbits of possibly different dimensions. However, the following result \aa{summarizing \cite[Propositions 5.2 and 6.6]{Agazzi}} says each $\mathcal{M}$ contains exactly one $\mathcal V$-orbit
 that has full measure, and is therefore open\footnote{Smooth submanifolds of $\mathcal{M}$ with nonzero codimension have volume measure zero. So any $\mathcal V$-orbit in $\mathcal{M}$ with nonzero measure must have the same dimension as, and therefore be open in, $\mathcal{M}$.}, in $\mathcal{M}$.

\begin{proposition}\label{prop:generic}\label{lem:XFullInM}
For every $\mathcal{M}_{Lorenz}(x)$ there exists a unique $\mathcal{V}_{Lorenz}$-orbit $\mathcal{X}$ with full measure in $\mathcal{M}_{Lorenz}(x)$, i.e. $\vol(\mathcal{X})=1$ where $\vol$ is the normalized volume form on $\mathcal{M}_{Lorenz}(x)$. Similarly, for every $\mathcal{M}_{Euler}(q)$ there exists a unique $\mathcal{V}_{Euler}$-orbit with full measure on $\mathcal{M}_{Euler}(q)$.
\end{proposition}

\begin{definition}[Generic orbit]\label{def:generic}
	The unique full-measure $\mathcal V$-orbits in $\mathcal{M}_{Lorenz}(x)$ and $\mathcal{M}_{Euler}(q)$ in \Cref{prop:generic} are called {\normalfont generic $\mathcal{V}_{Lorenz}$-orbits} and {\normalfont generic $\mathcal{V}_{Euler}$-orbits}, respectively. When there is no confusion, we also simply call them {\normalfont generic orbits}.
\end{definition}

The next result, \aa{summarizing \cite[Corollaries 5.3 and 6.7]{Agazzi}}, highlights a key feature of generic orbits, namely that the $\mathcal{V}_{Lorenz}$ and $\mathcal{V}_{Euler}$ splittings are ergodic on them. This follows from the fact that the Lie bracket condition holds on generic orbits\footnote{For both Lorenz and Euler the splitting vector fields alone span tangent spaces on generic orbits, i.e. there is no need for Lie brackets \cite[Proposition 5.2 and Lemma 6.15]{Agazzi}. So both models are ``elliptic" rather than just ``hypoelliptic."} of Lorenz and Euler and that Lebesgue measure in ambient space -- and therefore its disintegration onto generic orbits -- is $P_h$-invariant for all $h>0$.

\begin{theorem}\label{thrm:models_ergodic}
For any generic orbit $\mathcal{X}$ of either the Lorenz or Euler random splittings, there exists a unique $P_h$-invariant measure on $\mathcal{X}$ for every $h>0$. Furthermore, this measure is absolutely continuous with respect to the volume form $\vol$ on $\mathcal X$ and, respectively, on $\mathcal{M}_{Lorenz}(x)$ or $\mathcal{M}_{Euler}(q)$.
\end{theorem}

\noindent \Cref{thrm:models_ergodic} requires the assumption in \Cref{sec:splitting} that the support of the time distribution $\rho$ contains an interval \aa{$[0,\epsilon)$}. Note \Cref{cond:bound} holds for both models by \Cref{lem:bound} since $\mathcal{V}$-orbits are bounded and the splitting vector fields are smooth. By the discussion in \Cref{sec:exponents} the Lyapunov exponents of any such splitting on a generic orbit exist and are constant for $\vol$-almost every $x$ and almost every sequence of times $\tau$.

\begin{remark}\label{rmk:invariant}
For both Lorenz and Euler, not only is Lebesgue measure in the ambient
space, and therefore its disintegration onto $\mathcal{V}$-orbits (intended in the classical measure-theoretic sense),
  invariant under $P_h$ for any $h>0$, but it is also invariant under
  $\Phi_{h\tau}$ for every $\tau$. This follows from
  the fact that the $\{V_j\}_{j=1}^n$ are divergence free in both models, and thus each $\phi_t^{(i)}$ preserves Lebesgue measure in the ambient space for every $t$.
\end{remark}

\section{Main result on Lyapunov exponents of general random splittings}\label{sec:main}


Consider as in \Cref{sec:prelims} a random splitting
$\{\Phi^m_{h\tau}\}$ on a $d$-dimensional $\mathcal{V}$-orbit
$\mathcal{X}$ with $\mathcal{V}=\{V_j\}_{j=1}^n$ a family of complete,
analytic vector fields on $\mathbb{R}^D$. Assume $\mu$ is an
absolutely continuous\footnote{\textit{Absolutely continuous} means
  absolutely continuous with respect to $\vol$ on $\mathcal{X}$ unless
  otherwise specified.} measure on $\mathcal{X}$ that is ergodic and
invariant with respect to $P_h$ for every $h>0$. Before stating the main result of this section,
\Cref{thrm:main}, we need two additional definitions. First, for every
positive integer $m$ define the \textit{pushforward} of $\mu$ by
$\Phi^m_{h\tau}$ to be the probability measure $\mu_m\coloneqq(\Phi^m_{h\tau})_\#\mu$ on
$\mathcal{X}$ given by
\begin{align*}
	\mu_m(B) &\coloneqq (\Phi^m_{h\tau})_\#\mu(B)
		\coloneqq \mu\left((\Phi^m_{h\tau})^{-1}(B)\right).
\end{align*}
Second, for probability measures $\nu$ and $\mu$ on $\mathcal{X}$
define the \textit{relative entropy\footnote{Relative entropy is also often called \textit{Kullback-Leibler divergence}.} of $\nu$ with respect to $\mu$} by
\begin{align*}
	D_{KL}(\nu\parallel\mu) &\coloneqq
		\begin{cases}
			\int_\mathcal{X} \frac{d\nu}{d\mu}(x)\log\left(\frac{d\nu}{d\mu}(x)\right) \mu(dx) & \text{if $\nu\ll\mu$} \\
			\infty & \text{otherwise},
		\end{cases}
\end{align*}
where $\nu\ll\mu$ means $\nu$ is absolutely continuous with respect to
$\mu$ and $d\nu/d\mu$ is the Radon-Nikodym derivative. Note $\mu_m$ is
random since it depends on $\tau$, so we can consider the
\textit{average relative entropy},
$\mathbb{E}D_{KL}(\mu_m\parallel\mu)$. The finiteness of this quantity is important in what follows.
 Whenever $\mu_m=\mu$ almost surely, as in both the Lorenz
 and Euler models, we have $\mathbb{E}D_{KL}(\mu_m\parallel\mu)=\mathbb{E}D_{KL}(\mu\parallel\mu)=0$. Another condition guaranteeing $\mathbb{E}D_{KL}(\mu_m\parallel\mu)<\infty$ is that $\log(d\mu/d\vol)\in L^1(\mu)$ since this and \eqref{eq:int_cond} together imply $\mathbb{E}D_{KL}(\mu_m\parallel\mu)=-m\lambda_\Sigma(h)$ \cite[Theorem 4.2]{Baxendale}.

\begin{theorem}\label{thrm:main}
Assume \Cref{cond:bound}, $\mathbb{E}D_{KL}(\mu_m\parallel\mu)<\infty$ for all $m$, and the Lie bracket condition holds at $x_*\in\supp(\mu)$. Then there exists $h_*>0$ such that for all $h\in(0,h_*)$, if $d\lambda_1(h)=\lambda_\Sigma(h)$ then there is an open set $U$ in $\mathcal{X}$ such that $\mu(U)=1$, \om{$\Phi^m_t(U)\subseteq U$ for all $m$ and $t\in\mathbb{R}^{mn}_{\geq 0}$}, and either

\begin{alternative}\label{alt:conformal}
There is a Riemannian structure\footnote{ A \textit{Riemannian structure on $U$} is a family $\{g_x:x\in U\}$ of inner products $g_x:T_x\mathcal{X}\times T_x\mathcal{X}\to\mathbb{R}$ on $T_x\mathcal{X}$.} $\{g_x:x\in U\}$ on $U$ such that
\begin{align}\label{eq:conformal}
	g_{\Phi^m_t(x)}(D_x\Phi^m_t(x)\eta,D_x\Phi^m_t(x)\xi) &= \om{(\det D_x\Phi^m_t(x))^{2/d}}g_x(\eta,\xi)
\end{align}
for all $m$, \aa{$t \in \mathbb R_{\geq 0}^{mn}$}, $x\in U$, and $\eta, \xi \in T_x\mathcal{X}$.
That is, $\Phi^m_t$ is conformal with respect to $\{g_x:x\in U\}$.
\end{alternative}

\begin{alternative}\label{alt:subspaces}
For every $x\in U$ there exist proper linear subspaces $E^1_x,\dots,E^p_x$ of $T_x\mathcal{X}$ such that
\begin{align}\label{eq:subspaces}
	D_x\Phi^m_t(x)(E^i_x) &= E^{\sigma(i)}_{\Phi^m_t(x)}
\end{align}
for all $m$, $t \in \aa{\mathbb R_{\geq 0}^{mn}}$, and $i$, where $\sigma$ is a permutation possibly depending on $m$, $t$, and $x$.
\end{alternative}
\end{theorem}


\Cref{thrm:main} relies on \Cref{thrm:baxendale} and \Cref{prop:dense,prop:split_feller}, all stated below. We present its proof now however to highlight how these separate results combine to give our main one.

\begin{proof}[Proof of \Cref{thrm:main}]
\om{Since \Cref{cond:bound} implies \eqref{eq:int_cond} (\Cref{lem:bound}), the Lyapunov exponents $\lambda_1(h)\geq\cdots\geq\lambda_d(h)$ exist for every $h>0$ and are almost surely constant. Since the Lie bracket condition holds at $x_*\in\supp(\mu)$, \Cref{prop:split_feller} says there exists $\kappa$ and an open neighborhood $U_0$ of $x_*$ on which $P^\kappa_h$ is strong Feller for all $h>0$. Keeping $x_*\in U_0$ while shrinking $U_0$ if necessary, we may assume its closure is compact. Then by \Cref{prop:dense} there exists $h_*>0$ such that assumption (b) of \Cref{thrm:baxendale} holds for all $h\in(0,h_*)$. The result follows by \Cref{thrm:baxendale}.}
\end{proof}

\Cref{thrm:baxendale} is \cite[Theorem
6.9]{Baxendale} adapted to our setting. Its statement requires some definitions. Recall $P_h$ acts on the space $\mathcal{B}_b(\mathcal{X})$ of bounded, measurable functions $f:\mathcal{X}\to\mathbb{R}$ via
\begin{align}\label{eq:Kernel}
	P_hf(x) & = \mathbb{E}\Big(f\big(\Phi_{h\tau}(x)\big)\Big)
	= \int_{\mathbb{R}^n_+}f\big(\Phi_{ht}(x)\big) \prod_{i=1}^n\rho(t_i)dt_1\cdots dt_n.
\end{align}
$P_h$ is \textit{Feller} if it maps $\mathcal{C}_b(\mathcal{X})$ into $\mathcal{C}_b(\mathcal{X})$ and \textit{strong Feller} if it maps $\mathcal{B}_b(\mathcal{X})$ into $\mathcal{C}_b(\mathcal{X})$, where $\mathcal{C}_b(\mathcal{X})$ is the space of bounded, continuous functions on $\mathcal{X}$. Similarly, $P_h$ is strong Feller on an open $U\subseteq\mathcal{X}$ if \om{$P_hf\vert_U\in\mathcal{C}_b(U)$ for all $f\in\mathcal{B}_b(\mathcal{X})$.} We also need to consider the dynamics induced by $\{\Phi^m_{h\tau}\}$ on the projective bundle $\mathcal{PX}$ of $\mathcal{X}$. Specifically, define the \textit{lifted splitting} on $P\mathcal{X}$ by
\begin{align}\label{eq:lift}
	\widetilde\Phi^m_{h\tau}(x,\eta) &\coloneqq \left(\Phi^m_{h\tau}(x), \frac{D_x\Phi^m_{h\tau}(x)\eta}{\lVert D_x\Phi^m_{h\tau}(x)\eta\rVert}\right)
\end{align}
where, by a slight abuse of notation, $\eta$ denotes both an element of the tangent space $T_x\mathcal{X}$ and its equivalence class\footnote{The \textit{projective space} $P_x\mathcal{X}$ at $x$ is the
  space of all lines in the tangent space $T_x\mathcal{X}$.} in
$P_x\mathcal{X}$ whenever $\eta\neq 0$. The transition
kernel of \eqref{eq:lift} is denoted $\widetilde P_h$.

\begin{theorem}\label{thrm:baxendale}
Fix $h>0$. Assume \eqref{eq:int_cond} and $\mathbb{E}D_{KL}(\mu_m\parallel\mu)<\infty$ for all $m$. Suppose there exists $\kappa\in\mathbb{N}$ and an open $U_0\subseteq\mathcal{X}$ with $U_0\cap\supp(\mu)\neq\emptyset$ such that
\begin{enumerate}[label={(\alph*)}]
\item $P^\kappa_h$ is strong Feller on $U_0$ and

\item $\{ \widetilde P^\kappa_hg\vert_{\mathcal{P}_x\mathcal{X}} : g\in\mathcal{C}_b(\mathcal{PX})\}$ is dense in $\mathcal{C}(\mathcal{P}_x\mathcal{X})$ for all $x\in U_0$.
\end{enumerate}
Then if $d\lambda_1(h)=\lambda_\Sigma(h)$ there exists an open $U\subseteq\mathcal{X}$ such that $\mu(U)=1$, \om{$\Phi^m_t(U)\subseteq U$ for all $m$ and $t \in \mathbb{R}^{mn}_{\geq 0}$}, and Alternative \ref{alt:conformal} or \ref{alt:subspaces} holds.
\end{theorem}

Theorem \ref{thrm:baxendale} derives largely from two relationships, one between Lyapunov exponents and relative entropy, and one between $\Phi^m_{h\tau}$ and $\widetilde\Phi^m_{h\tau}$. The assumption
$\mathbb{E}D_{KL}(\mu_m\parallel\mu)<\infty$ for all $m$ allows
for the comparison of $\lambda_1(h)$ and $\lambda_\Sigma(h)$ that lies at
the heart of Theorem \ref{thrm:baxendale}. In particular,
\cite[Corollary 5.6]{Baxendale} says that if \eqref{eq:int_cond} and the finite average relative entropy condition
hold, then $d\lambda_1(h)=\lambda_\Sigma(h)$ implies there exists a $\nu$ in
the space $\mathscr{P}_\mu(P\mathcal{X})$ of probability measures on
$P\mathcal{X}$ with $\mathcal{X}$-marginal $\mu$ whose \textit{regular conditional probability distributions\footnote{The $\nu_x$ are probability measures on $P\mathcal{X}$ which are well-defined for $\mu$-almost every $x$ and satisfy $\nu_x(P_x\mathcal{X})=1$.} (rcpd)} $\{\nu_x:x\in\mathcal{X}\}$ satisfy
\begin{align}\label{eq:icd}
	\mu\left\{x : \big(\widetilde{\Phi}^m_{h\tau}\big)_\#\nu_x = \nu_{\Phi^m_{h\tau}(x)}\ \text{for every}\ m\right\} &= 1\ \text{for almost every $\tau$}.
\end{align}

\begin{proof}[Proof of \Cref{thrm:baxendale}]
\om{There are only two slight differences between \Cref{thrm:baxendale} and \cite[Theorem 6.9]{Baxendale}. The first is our addition of $\kappa$ which in the latter theorem is $1$, i.e. conditions (a) and (b) are assumed to hold for $P_h$ rather than $P^\kappa_h$. Note however that the conclusions of \Cref{thrm:baxendale} -- in particular \Cref{alt:conformal,alt:subspaces} -- do not involve $\kappa$. This is because the only contribution of conditions (a) and (b) to \cite[Theorem 6.9]{Baxendale} is in proving \cite[Proposition 6.3]{Baxendale}, and their only role in that proof is to construct a continuous (on $U_0$) version of every $\nu\in\mathscr{P}_\mu(\mathcal{PX})$ satisfying \eqref{eq:icd}. Specifically, in the proof of \cite[Proposition 6.3]{Baxendale} just after Equation (6.2), we can replace the statement ``\textit{If $T=\mathbb{N}$ take $t=1$ and use the denseness part  of assumption (ii)}" to ``\textit{If $T=\mathbb{N}$ take $t=\kappa$ and use the denseness part  of assumption (ii)}" to obtain the desired continuous version of $\nu_x$ without loss of generality. Since the invariant measure $\mu$ and \eqref{eq:icd} do not depend on $\kappa$ (for the latter because \eqref{eq:icd} holds for all $m$), the measures $\nu\in\mathscr{P}_\mu(\mathcal{PX})$, their rcpds, and continuity of $x\mapsto \nu_x$ do not depend on $\kappa$ as well. Thus the rest of the proof of \cite[Theorem 6.9]{Baxendale} goes through unchanged and is independent of $\kappa$.

The second slight difference is that the conclusions of \Cref{thrm:baxendale} (and hence of \Cref{thrm:main}) -- including \Cref{alt:conformal,alt:subspaces} -- are stated in terms of $\Phi^m_t$ for all $m$ and $t\in\mathbb{R}^{mn}_{\geq 0}$. Meanwhile those of \cite[Theorem 6.9]{Baxendale} are stated to hold for $\Phi^m_t\in S$ with $S$ defined as follows (see also \cite[Def.~6.1]{Baxendale}). For $h>0$ and $m\in\mathbb{N}$, $\Phi^m_h:\tau\mapsto \Phi^m_{h\tau}$ is a random map from $\supp(\rho_h^{\otimes mn})$ into the space $\mathcal{C}^{\infty}(\mathcal{X},\mathcal{X})$ of smooth diffeomorphisms of $\mathcal{X}$. Define $S_m$ to be the support of the distribution of $\Phi^m_\tau$ in $\mathcal{C}^{\infty}(\mathcal{X},\mathcal{X})$, namely $S_m=\{\Phi^m_t:t\in\supp(\rho_h^{\otimes mn})\}$. Then $S$ is defined to be the closure of $\cup_{m=1}^\infty S_m$ in $\mathcal{C}^{\infty}(\mathcal{X},\mathcal{X})$. We now show $S=\{\Phi^m_t:m\in\mathbb{N}, t\in\mathbb{R}^{mn}_{\geq 0}\}$ for every $h>0$, proving the conclusions of \Cref{thrm:baxendale} are in fact equivalent to those in \cite[Theorem 6.9]{Baxendale}. Fix $h>0$. Clearly $S\subseteq\{\Phi^m_t:m\in\mathbb{N}, t\in\mathbb{R}^{mn}_{\geq 0}\}$. For the reverse inclusion, fix $m$ and $t\in\mathbb{R}^{mn}_{\geq 0}$ and consider $\Phi^m_t=\varphi^{(n)}_{t_{mn}}\circ\cdots\varphi^{(1)}_{t_1}$. Since $\supp(\rho)$ contains $(0,\epsilon)$ for some $\epsilon>0$, for each $i \in [mn]$ and component $t_i$ of $t$ there exists $\alpha_i$, $s^{(i)} = (s^{(i)}_1, \dots, s^{(i)}_{\alpha_in}) \in \supp(\rho_h^{\otimes \alpha_in})$ with all $s^{(i)}_j = 0$
	except  for $j\equiv i \mod n$, and with $\sum_j h s^{(i)}_j = t_i$. Therefore $\phi_{t_i}^{(i)}=\Phi_{hs^{(i)}}^{\alpha_i}\in S_{\alpha_i}$ and hence $\Phi^m_t=\Phi_{hs^{(mn)}}^{\alpha_{mn}}\circ\cdots\circ \Phi_{hs^{(1)}}^{\alpha_{1}}\in S$.}
\end{proof}


\om{
Assumption (a) of \Cref{thrm:baxendale} is addressed in \Cref{sec:feller}. The remainder of this section is dedicated to proving \Cref{prop:dense} which says that \Cref{cond:bound} implies assumption (b) in \Cref{thrm:baxendale}.

\begin{lemma}\label{lem:compact}
Let $K$ be compact and $f:K\times \mathbb{R}_+\to \mathbb{R}$ be continuous. If $f(x,0)=c$ for all $x\in K$ then for every $\epsilon>0$ there exists $h_*>0$ such that $f(x,h)\in (c-\epsilon,c+\epsilon)$ for all $x\in K$ and $h\in [0,h_*)$.
\end{lemma}

\begin{proof}
The proof is by contradiction. If the result is false then there exists $\epsilon>0$ such that for all $n$ there exists $x_n\in K$ and $h_n\in(0,1/n)$ with $f(x_n,h_n)\notin (c-\epsilon,c+\epsilon)$. Since $K$ is compact we may assume $x_n$ converges to $x_*\in K$. By continuity and since $f(x,0)=c$ for all $x$,
\begin{align*}
	c &= f(x_*,0)
		= \lim_{n\to\infty} f(x_n,h_n)
		\notin (c-\epsilon,c+\epsilon). \tag*{\qed}
\end{align*}
\renewcommand{\qedsymbol}{}
\end{proof}

\begin{proposition}\label{prop:dense}
If \Cref{cond:bound} holds, then for every $\kappa\in\mathbb{N}$ and open subset $U$ of $\mathcal{X}$ with compact closure $\overline{U}$ and $U\cap\supp(\mu)\neq\emptyset$ there exists $h_*>0$ such that
\begin{align*}
	\left\{ \widetilde P^\kappa_hg\vert_{\mathcal{P}_x\mathcal{X}} : g\in\mathcal{C}_b(\mathcal{PX})\right\}
\end{align*}
is dense in $\mathcal{C}(\mathcal{P}_x\mathcal{X})$ for all $h\in(0,h_*)$ and $x\in U_0$.
\end{proposition}

\begin{proof}
Points in $\mathcal{P}_x\mathcal{X}$ are represented by points in $S^{d-1}$. Define the off-diagonal $\mathcal{D}\coloneqq\{(\eta,\eta')\in S^{d-1}\times S^{d-1} : \eta\neq\eta'\}$ and $\mathcal{G}\coloneqq\{g_{\eta,\eta'} : (\eta,\eta')\in\mathcal{D}\}$ with $g_{\eta,\eta'}:\mathcal{PX}\to\mathbb{R}$ given by
\begin{align*}
	g_{\eta,\eta'}(x,\xi) &= \lVert\eta-\eta'\rVert^{-2}\langle\eta-\eta', \xi\rangle.
\end{align*}
Note $\mathcal{G}\subseteq \mathcal{C}_b(\mathcal{PX})$. Fix $\kappa\in\mathbb{N}$ and define $f:\overline U\times\mathcal{D}\times\mathbb{R}_+\to\mathbb{R}$ by
\begin{align}\label{eq:f}
	f(x,\eta,\eta',h) &= \left\lvert\widetilde P^\kappa_hg_{\eta,\eta'}(x,\eta) - \widetilde P^\kappa_hg_{\eta,\eta'}(x,\eta')\right\rvert.
\end{align}
We will show there exists $h_*>0$ such that
\begin{align}\label{eq:goal}
	\inf\left\{f(x,\eta,\eta',h) : x\in \overline U, (\eta,\eta')\in\mathcal{D}, h\in[0,h_*)\right\} &\geq \tfrac{1}{2}.
\end{align}
\Cref{prop:dense} then immediately follows by the Stone-Weierstrass theorem. Define $A_{h,x}\coloneqq D_x\Phi^\kappa_{h\tau}(x)$ and $A_{h,x,\eta}\coloneqq A_{h,x}/\lVert A_{h,x}\eta\rVert$. By \Cref{lem:bound} there exists $h_1>0$ and $C<\infty$ such that
\begin{align}\label{eq:h1}
	\sup\left\{\mathbb{E}\lVert A_{h,x}^{-1}\rVert^{m_1}\lVert A_{h,x}\rVert^{m_2}\lVert \partial_h A_{h,x}\rVert^{m_3} : x\in\mathcal{X}, h\in [0,h_1)\right\} &\leq C
\end{align}
when $(m_1,m_2,m_3)=(1,1,0)$ and $(m_1,m_2,m_3)=(2,1,1)$. Adding and subtracting $\mathbb{E}A_{h,x,\eta}\eta'$,
\begin{equation}\label{eq:diff}
\begin{aligned}
	\widetilde P^\kappa_hg_{\eta,\eta'}(x,\eta) - \widetilde P^\kappa_hg_{\eta,\eta'}(x,\eta') &= \lVert\eta-\eta'\rVert^{-2}\langle\eta-\eta', \mathbb{E}A_{h,x,\eta}\eta-\mathbb{E}A_{h,x,\eta'}\eta'\rangle \\
		&= \lVert\eta-\eta'\rVert^{-2}\langle\eta-\eta', \mathbb{E}A_{h,x,\eta}(\eta-\eta')\rangle \\
		&\qquad + \lVert\eta-\eta'\rVert^{-2}\langle\eta-\eta',(\mathbb{E}A_{h,x,\eta}-\mathbb{E}A_{h,x,\eta'})\eta'\rangle \\
		&= T_1(x,\eta,\eta',h) + T_2(x,\eta,\eta',h)
\end{aligned}
\end{equation}
where $T_1$ and $T_2$ denote the first and second terms in the third expression, respectively. Since $\mathbb{E}\lvert\langle\xi, A_{h,x,\eta}\xi\rangle\rvert \leq \mathbb{E}\lVert A_{h,x}^{-1}\rVert\lVert A_{h,x}\rVert< \infty$ by \eqref{eq:h1}, the dominated convergence theorem implies $f_1:\overline U\times S^{d-1}\times S^{d-1}\times[0,h_1)\to\mathbb{R}$ by $f_1(x,\eta,\xi,h)=\mathbb{E}\langle\xi,A_{h,x,\eta}\xi\rangle$ is continuous. Also $f_1(x,\eta,\xi,0)=1$ for all $(x,\eta,\xi)\in\overline U\times S^{d-1}\times S^{d-1}$. Shrinking $h_1$ if necessary, \Cref{lem:compact} implies
\begin{align}\label{eq:T1}
	\inf\left\{T_1(x,\eta,\eta',h) : x\in \overline U, (\eta,\eta')\in\mathcal{D}, h\in[0,h_1)\right\} \geq \tfrac{3}{4}.
\end{align}
For $T_2$ we have
\begin{align*}
	A_{h,x,\eta}-A_{h,x,\eta'} &= \frac{(\lVert A_{h,x}\eta'\rVert-\lVert A_{h,x}\eta\rVert)A_{h,x}}{\lVert A_{h,x,\eta}\rVert\lVert A_{h,x,\eta'}\rVert}.
\end{align*}
So for any $h\in(0,h_1)$,
\begin{align*}
	\lvert T_2\rvert &\leq \lVert\eta-\eta'\rVert^{-2}\mathbb{E}\frac{\left\lvert\lVert A_{h,x}\eta'\rVert-\lVert A_{h,x}\eta\rVert\right\rvert}{\lVert A_{h,x,\eta}\rVert\lVert A_{h,x,\eta'}\rVert}\left\lvert\langle\eta-\eta',A_{h,x}\eta'\rangle\right\rvert \\
		&\leq \lVert\eta-\eta'\rVert^{-1}\mathbb{E}\lVert A_{h,x}^{-1}\rVert^2\lVert A_{x,h}\rVert\lvert\langle\eta-\eta',A_{h,x}\eta'\rangle\rvert \\
		&\leq \left\lvert\left\langle\frac{\eta-\eta'}{\lVert\eta-\eta'\rVert},\eta'\right\rangle\right\rvert\mathbb{E}\lVert A_{h,x}^{-1}\rVert^2\lVert A_{x,h}\rVert
		+ h\mathbb{E}\lVert A_{h,x}^{-1}\rVert^2\lVert A_{x,h}\rVert\lVert\partial_hA_{h,x}\mid_{h=\hat h}\rVert \\
		&\leq \left\lvert\left\langle\frac{\eta-\eta'}{\lVert\eta-\eta'\rVert},\eta'\right\rangle\right\rvert\mathbb{E}\lVert A_{h,x}^{-1}\rVert^2\lVert A_{x,h}\rVert+ Ch.
\end{align*}
The second inequality follows by the reverse triangle inequality and $\lVert A_{h,x,\eta}\rVert^{-1}\lVert A_{h,x,\eta'}\rVert^{-1}\leq \lVert A_{x,h}^{-1}\rVert^2$. The third is by Taylor's theorem and $A_{0,x}=I$ which give $\langle\eta-\eta',A_{h,x}\eta'\rangle = \langle\eta-\eta',\eta'\rangle + h\langle\eta-\eta',\partial_hA_{h,x}\mid_{h=\hat h}\eta'\rangle$ for some $\hat h\in(0,h)$, and the last inequality follows from \eqref{eq:h1}. Since $\lim_{\lVert\eta-\eta'\rVert\to 0} \lVert\eta-\eta'\rVert^{-1}\langle\eta-\eta',\eta'\rangle=0$ for all $\eta,\eta'\in S^{d-1}$ there exists $h_2>0$ and $\epsilon>0$ such that
\begin{align}\label{eq:T2}
	\sup\left\{T_2(x,\eta,\eta',h) : x\in \overline U, (\eta,\eta')\in\mathcal{D}_\epsilon, h\in[0,h_2)\right\} &\leq \tfrac{1}{4}
\end{align}
where $\mathcal{D}_\epsilon=\{(\eta,\eta')\in\mathcal{D} : \lVert\eta-\eta'\rVert<\epsilon\}$. Setting $h_3=\min\{h_1,h_2\}$,
\begin{align*}
	\inf_{x\in \overline U,(\eta,\eta')\in\mathcal{D}_\epsilon,h\in[0,h_3)} f(x,\eta,\eta',h) &= \inf_{x\in \overline U,(\eta,\eta')\in\mathcal{D}_\epsilon,h\in[0,h_3)} \lvert T_1+T_2\rvert \\
		&\geq \inf_{x\in \overline U,(\eta,\eta')\in\mathcal{D},h\in[0,h_3)}T_1 - \sup_{x\in \overline U,(\eta,\eta')\in\mathcal{D}_\epsilon,h\in[0,h_3)} \lvert T_2\rvert \\
		&\geq \tfrac{1}{2}.
\end{align*}
The first equality is by \eqref{eq:f} and \eqref{eq:diff}. The first inequality follows from $\lvert T_1+T_2\rvert\geq T_1-\lvert T_2\rvert$ and the last from \eqref{eq:T1} and \eqref{eq:T2}. This proves \eqref{eq:goal} with $\mathcal{D}$ replaced by $\mathcal{D}_\epsilon$. To extend to $\mathcal{D}$, observe $\mathcal{D}_\epsilon^c$ is a closed subset of the compact space $S^{d-1}\times S^{d-1}$ and therefore compact. Since $f(x,\eta,\eta',0)=1$ for all $\eta\neq\eta'$, \Cref{lem:compact} implies there exists $h_4>0$ such that
\begin{align*}
	\inf\left\{f(x,\eta,\eta',h) : x\in \overline U, (\eta,\eta')\in\mathcal{D}_\epsilon^c, h\in[0,h_4)\right\} &\geq \tfrac{1}{2}.
\end{align*}
Setting $h_*=\min\{h_3,h_4\}$ yields \eqref{eq:goal} and concludes the proof.
\end{proof}
}


\section{Transition densities and strong Feller}\label{sec:feller}


We begin this section by considering a more general setting than
random splitting. \Cref{cor:density} and \Cref{thrm:feller} are stated
at this heightened level of generality. We then apply these results to
random splitting to prove that if the Lie bracket condition holds at a
point, then there is a $\kappa\in\mathbb{N}$ and a neighborhood $U_0$ of that point such that the transition kernel $P_h^\kappa$ has a density and
is strong Feller on $U_0$ (\Cref{prop:split_feller}). Only the
strong Feller part of this result is needed in our study of Lyapunov
exponents; the existence of transition densities is included because
it is a direct consequence of the coarea formula and may be of
independent interest. To avoid notational confusion, $\mathcal{X}$
will always denote a $\mathcal V$-orbit of random splitting.


\subsection{Transition densities in a general setting}\label{sec:general}


Let $\mathcal{Y}$ be a smooth $d$-dimensional manifold with volume form $\vol_\mathcal{Y}$, let $\Omega$ be
a connected, open subset of $\mathbb{R}^n$ with $n\geq d$, and let
$\rho$ be a probability measure on $\Omega$ that is absolutely
continuous with respect to Lebesgue measure. Any continuous
$\Psi:\mathcal{Y}\times\Omega\to\mathcal{Y}$ induces an operator $P:\mathcal{B}_b(\mathcal{Y})\to\mathcal{B}_b(\mathcal{Y})$ given by
\begin{align}\label{eq:kernel}
	Pf(y) &\coloneqq \mathbb{E}\big(f(\Psi(y,\omega))\big)
		\coloneqq \int_\Omega (f\circ\Psi)(y,\omega)\rho(\omega)d\omega
\end{align}
where, as before, $\rho$ denotes both the measure and its density. $P$ has a \textit{density} with respect to $\vol_\mathcal{Y}$ if there exists an integrable function $p:\mathcal{Y}\times\mathcal{Y}\to \mathbb{R}_{\geq 0}$ satisfying
\begin{align*}
	Pf(y) &= \int_\mathcal{Y} f(\hat y)p(y,\hat y)\vol_{\mathcal{Y}}(d\hat y).
\end{align*}
The existence of a density for $P$ is a direct corollary of the coarea formula when $\Psi$ is a submersion.

\begin{lemma}[Coarea formula]\label{lem:Coarea}
Let $\mathcal{Y}$ and $\mathcal{Z}$ be smooth $d$ and $n$-dimensional manifolds with volume forms $\vol_\mathcal{Y}$ and $\vol_\mathcal{Z}$, respectively. Suppose $F:\mathcal{Z}\to\mathcal{Y}$ is a $C^1$ submersion, i.e. $DF(z):T_z\mathcal{Z}\to T_{F(z)}\mathcal{Y}$ is surjective for every $z$ in $\mathcal{Z}$. Then for any $f:\mathcal{Z}\to\mathbb{R}$ measurable with respect to $\vol_\mathcal{Z}$,
\begin{align}\label{eq:Coarea}
	\int_\mathcal{Z}	 f(z) \vol_\mathcal{Z}(dz) &= \int_\mathcal{Y}\bigg(\int_{F^{-1}(y)} \frac{f(z)}{\sqrt{\det DF(z)DF(z)^*}} \mathcal{H}^{n-d}(dz)\bigg) \vol_\mathcal{Y}(dy)
\end{align}
as long as either side is finite. Here $\mathcal{H}^{n-d}(dz)$ is $n-d$-dimensional Hausdorff measure on $\mathcal{Z}$.
\end{lemma}

\begin{proof}
See \cite[Corollary 2.2]{Nicolaescu} or \cite[Theorem 3.2.11]{Federer69}.
\end{proof}

\begin{cor}\label{cor:density}
Suppose $\Psi\in C^1(\mathcal{Y}\times\Omega, \mathcal{Y})$. If for every $y\in \mathcal{Y}$,
\begin{align}\label{eq:surjective_as}
	D_\omega\Psi(y,\omega) : T_\omega\Omega \to T_{\Psi(y,\omega)}\mathcal{Y}
\end{align}
is surjective for Lebesgue-almost every $\omega$, then $P$ has density $p$ with respect to $\vol_{\mathcal{Y}}$ given by
\begin{align}\label{eq:density}
	p(y,\hat y) &= \int_{\Psi_y^{-1}(\hat y)}
		\frac{\rho(\omega)}{\sqrt{\det D_\omega\Psi(y,\omega)D_\omega\Psi(y,\omega)^*}} \mathcal{H}^{n-d}(d\omega)
\end{align}
where $\Psi_y^{-1}(\hat y)\coloneqq\{\omega:\Psi(\omega,y)=\hat y\}$.
\end{cor}

The assumption that $D_\omega\Psi(y,\omega)$ is surjective for almost every $\omega$ is equivalent to
\begin{align}\label{eq:malliavin}
	M(y,\omega) &\coloneqq D_\omega\Psi(y,\omega)D_\omega\Psi(y,\omega)^*
\end{align}
being almost-surely invertible. Since $\mathcal{Y}$ is $d$-dimensional, $M$ is a symmetric, positive-semidefinite $d\times d$ matrix. Note $n\geq d$ is a necessary condition for its invertibility, which is why we made this assumption above. $M$ is analogous to the controllability Gramian matrix in control theory where $\omega$ is a control, and the Malliavin matrix in Malliavin calculus where randomness is Brownian motion. $\sqrt{\det M(y,\omega)}$ is an expression of the tangential Jacobian, so in this setting the coarea formula is just a generalization of the classical change of variable formula \cite{Morgan16, Maggi_2012}.

\begin{remark}
The assumption in \cref{cor:density} that for fixed $y$ the map $\Psi(y,\cdot):\Omega\to\mathcal{Y}$ from noise space to state space is a submersion or, equivalently, that $M$ is invertible, is a form of hypoellipticity. It allows
         some regularity of the noise distribution $\rho$ to be
         transferred to the state space.
\end{remark}

\begin{proof}[Proof of \Cref{cor:density}]
For $f\in\mathcal{B}_b(\mathcal{Y})$ we have
	\begin{align*}
		Pf(y) &= \int_\Omega f(\Psi(y,\omega))\rho(\omega)d\omega
					= \int_\mathcal{Y}\bigg(\int_{\Psi^{-1}_y(\hat y)} \frac{f(\Psi(y,\omega))\rho(\omega)}{\sqrt{\det M(y,\omega)}} \mathcal{H}^{n-d}(d\omega)\bigg) \vol_\mathcal{Y}(d\hat y) \\
					&= \int_\mathcal{Y}f(\hat y)\bigg(\int_{\Psi^{-1}_y(\hat y)} \frac{\rho(\omega)}{\sqrt{\det M(y,\omega)}} \mathcal{H}^{n-d}(d\omega)\bigg) \vol_\mathcal{Y}(d\hat y)
					= \int_\mathcal{Y} f(\hat y)p(y,\hat y)\vol_\mathcal{Y}(d\hat y).
	\end{align*}
where the second equality is the coarea formula, the third holds because $f(\Psi(y,\omega))=f(\hat y)$ on $\Psi^{-1}_y(\hat y)$, and the fourth is by definition of $p(y,\hat y)$. One caveat in our application of the coarea formula is that our assumption says $\det M(y,\omega)$ is nonzero only almost-surely. This is not an issue however since $\{\hat y:\Psi(y,\omega)=\hat y \text{ and } \det M(y,\omega)=0\}$ has measure 0 in $\mathcal{Y}$. Therefore we can define $p(y,\hat y)=0$ for $\hat y$ in this set without changing the value of the integral of $f$ against $p$.
\end{proof}


\subsection{Strong Feller in a general setting}\label{sec:SFgeneral}


Let $\mathcal{Y}$, $\Omega$, $\rho$, and $\Psi$ be as in \Cref{sec:general}. Recall from \Cref{sec:main} the transition kernel $P$ defined in \eqref{eq:kernel} is strong Feller on an open $U_0\subseteq\mathcal{Y}$ if $Pf\vert_{U_0}\in\mathcal{C}_b(U_0)$ whenever $f\in\mathcal{B}_b(\mathcal{Y})$. Being strong Feller implies $P$ has a regularizing effect.

\begin{proposition}\label{thrm:feller}
Suppose $U_0\subseteq\mathcal{Y}$ is open and $\Psi\in C^1(\mathcal{Y}\times\Omega, \mathcal{Y})$. If $M(y,\omega)$ is invertible for every $y \in U_0$ and Lebesgue-almost every $\omega\in\Omega$, then $P$ is strong Feller on $U_0$.
\end{proposition}

\noindent The next lemma is used in the proof of \Cref{thrm:feller}.

\begin{lemma}\label{lem:LowerBound}
Suppose $\Psi\in C^1(\mathcal{Y}\times\Omega, \mathcal{Y})$, $U_0\subseteq\mathcal{Y}$ is open, and $K\subseteq\Omega$ is compact. For $y \in \mathcal{Y}$ define
	\begin{align*}
		A_K(y) & \coloneqq \{\omega\in K : \det M(y,\omega)=0\}.
	\end{align*}
If $M(y,\cdot)$ is invertible for every $y\in U_0$ and almost every $\omega\in\Omega$, then for any $y_* \in U_0$ and $\epsilon>0$ there exists an open $W\subseteq\Omega$ and a $\delta>0$ such that $\rho(W)<\epsilon$, $B_\delta(y_*)\subseteq U_0$, and $A_K(y)\subseteq W$ for all $y \in B_\delta(y_*)$. Moreover, there exists an open neighborhood $W'$ of $K\cap W^c$ such that
\begin{align}\label{eq:LowerBound1}
	\inf\left\{\det M(y,\omega) : y\in B_\delta(y_*), \omega\in W'\right\} & > 0.
\end{align}
\end{lemma}

\begin{proof}
Fix $y_* \in U_0$ and $\epsilon>0$. By assumption, $A_K(y)$ has Lebesgue measure zero for all $y$. So since $\rho$ is absolutely continuous there exists an open neighborhood $W$ of $A_K(y_*)$ such that $\rho(W)<\epsilon$. Suppose toward a contradiction there is a sequence $\{y_n\}\subseteq U_0$ converging to $y_*$ with $A_K(y_n)$ not contained in $W$; that is, for each $y_n$ there exists $\omega_n\in K\cap W^c$ satisfying $\det M(y_n,\omega_n)=0$. Then since $K\cap W^c$ is compact there is a subsequence $\{\omega_{n_k}\}$ that converges to some $\omega_*\in K\cap W^c$. Since $(y,\omega)\mapsto\det M(y,\omega)$ is continuous,
\begin{align*}
	0 & = \lim_{k\to\infty} \det M(y_{n_k},\omega_{n_k})
		= \det M(y_*,\omega_*).
\end{align*}
But this implies $\omega_*\in A_K(y_*)\cap(K\cap W^c)=\emptyset$, a contradiction. For the ``moreover'' part of the lemma, it follows from preceding argument that there exists $\delta>0$ such that $B_\delta(y_*)\subseteq U_0$ and
\begin{align*}
	\inf\left\{\det M(y,\omega) : y\in B_\delta(y_*), \omega\in K\cap W^c\right\} \geq 2\alpha
\end{align*}
for some $\alpha>0$. Set $F_y\coloneqq \det M(y,\cdot)$ and let $K'$ be the closure in $\Omega$ of
\begin{align*}
	\bigcup_{y\in B_\delta(y_*)} F_y^{-1}(0,\alpha).
\end{align*}
Straightforward verification shows $K\cap W^c$ and $K'$ are closed and disjoint and are therefore separated by disjoint open sets. By continuity and construction, any such neighborhood $W'$ of $K\cap W^c$ satisfies
\begin{align*}
	\inf\left\{\det M(y,\omega) : y\in B_\delta(y_*), \omega\in W'\right\} & \geq \alpha
		> 0. \qedhere
\end{align*}
\end{proof}

\begin{proof}[Proof of \Cref{thrm:feller}]
Fix $f \in \mathcal{B}_b(\mathcal{Y})$. The case $f\equiv 0$ is immediate, so assume otherwise. Set $C\coloneqq 6\lVert f\rVert_\infty>0$ and fix $y_* \in U_0$ and $\epsilon>0$. We will show there exists $\delta>0$ such that $\lvert Pf(y)-Pf(y_*)\rvert<\epsilon$ for all $y\in B_\delta(y_*)\subseteq U_0$, proving $Pf\vert_{U_0}$ is continuous. Since $\rho$ is a Borel measure it is tight. So there exists a compact $K\subseteq\Omega$ such that $\rho(K^c)<\epsilon/C$. By Lemma \ref{lem:LowerBound} there exists $W\subseteq\Omega$ open and $\delta_1>0$ such that $\rho(W)<\epsilon/C$, $B_{\delta_1}(y_*)\subseteq U_0$, and
\begin{align}\label{eq:LowerBound2}
	\inf\left\{\det M(y,\omega) : y\in B_{\delta_1}(y_*), \omega\in W'\right\} & \eqqcolon \alpha
		> 0
\end{align}
for some open neighborhood $W'$ of $K\cap W^c$. Now since $\1_{K^c}+\1_{K\cap W}+\1_{K\cap W^c}=1$,
	\begin{multline}\label{eq:EpsilonBound1}
		Pf(y)-Pf(y_*)  = \mathbb{E}\Big(\big[f\big(\Psi(y)\big)-f\big(\Psi(y_*)\big)\big]\1_{K^c}\Big) + \mathbb{E}\Big(\big[f\big(\Psi(y)\big)-f\big(\Psi(y_*)\big)\big]\1_{K\cap W}\Big) \\
		                + \mathbb{E}\Big(\big[f\big(\Psi(y)\big)-f\big(\Psi(y_*)\big)\big]\1_{K\cap W^c}\Big)
	\end{multline}
for any $y$ in $\mathcal{Y}$. By our choice of $K$,
\begin{align}\label{eq:EpsilonBound2}
		\Big\lvert \mathbb{E}\Big(\big[f\big(\Psi(y)\big)-f\big(\Psi(y_*)\Big)\big]\1_{K^c}\big) \Big\rvert & \leq 2\lVert f\rVert_\infty\rho(K^c)
		< \frac{\epsilon}{3}.
\end{align}
And by our choice of $W$,
\begin{align}\label{eq:EpsilonBound3}
	\Big\lvert\mathbb{E}\Big(\big[f\big(\Psi(y)\big)-f\big(\Psi(y_*)\big)\big]\1_{K\cap W}\Big) \Big\rvert & \leq 2\lVert f\rVert_\infty\rho(W)
		< \frac{\epsilon}{3}.
\end{align}
To handle the third expectation (the one with $\1_{K\cap W^c}$), set $L^1\coloneqq L^1(\vol_\mathcal{Y})$ and choose a compactly supported, continuous function $g$ on $\mathcal{Y}$ such that $\lVert g-f\rVert_{L^1} < \epsilon/(18\sqrt{\alpha})$. This can always be done since compactly supported, continuous functions are dense in $L^1$ \cite[Proposition 7.9]{Folland}. By adding and subtracting $g\circ\Psi$ appropriately,
\begin{equation}\label{eq:CoareaBound1}
\begin{aligned}
	\mathbb{E}\Big(\big[f\big(\Psi(y)\big)-f\big(\Psi(y_*)\big)\big]\1_{K\cap W^c}\Big) &= \mathbb{E}\Big(\big[f\big(\Psi(y)\big)-g\big(\Psi(y)\big)\big]\1_{K\cap W^c}\Big) \\
		&\qquad + \mathbb{E}\Big(\big[g\big(\Psi(y)\big)-g\big(\Psi(y_*)\big)\big]\1_{K\cap W^c}\Big) \\
		&\qquad\quad + \mathbb{E}\Big(\big[g\big(\Psi(y_*)\big)-f\big(\Psi(y_*)\big)\big]\1_{K\cap W^c}\Big).
\end{aligned}
\end{equation}
Since $g\circ\Psi$ is continuous, there exists $\delta_2>0$ such that for every $y\in B_{\delta_2}(y_*)$,
\begin{align}\label{eq:CoareaBound2}
	\Big\lvert \mathbb{E}\Big(\big[g\big(\Psi(y)\big)-g\big(\Psi(y_*)\big)\big]\1_{K\cap W^c}\Big)\Big\rvert &< \frac{\epsilon}{9}.
\end{align}
Set $\delta\coloneqq\min\{\delta_1,\delta_2\}$ and fix $y\in B_\delta(y_*)$. Setting $S(y,\hat{y})\coloneqq \{\omega\in W':\Psi(y,\omega)=\hat{y}\}$,
{\small
\begin{align*}
		\mathbb{E}\Big(\big[f\big(\Psi(y)\big)-g\big(\Psi(y)\big)\big]\1_{K\cap
  W^c}\Big)
			&= \int_{W'} \big[f\big(\Psi(y,\omega)\big)-g\big(\Psi(y,\omega)\big)\big]\1_{K\cap W^c}(\omega)\rho(\omega)d\omega \\
			&= \int_\mathcal{Y}\Big(\int_{S(y,\hat{y})}\big[f\big(\Psi(y,\omega)\big)-g\big(\Psi(y,\omega)\big)\big]\frac{\1_{K\cap W^c}(\omega)\rho(\omega)}{\sqrt{\det M(y,\omega)}}d\omega\Big)\vol_\mathcal{Y}(d\hat{y}) \\
			&= \int_\mathcal{Y}\big(f(\hat{y})-g(\hat{y})\big)\Big(\int_{S(y,\hat{y})}\frac{\1_{K\cap W^c}(\omega)\rho(\omega)}{\sqrt{\det M(y,\omega)}}d\omega\Big)\vol_\mathcal{Y}(d\hat{y}).
\end{align*}
}The first equality holds since $K\cap W^c\subseteq W'$. The second is the coarea formula which applies since $\Psi(y,:):W'\to\mathcal{Y}$ is a submersion for all $y\in B_\delta(y_*)$ by \eqref{eq:LowerBound2}. Moreover,
\begin{align*}
	\sup\left\{\frac{1}{\sqrt{\det M(y,\omega)}} : y\in B_\delta(y_*), \omega\in W'\right\} &\leq \frac{1}{\sqrt{\alpha}}.
\end{align*}
So our choice of $g$ implies that for all $y\in B_\delta(y_*)$,
\begin{align}\label{eq:CoareaBound3}
	\Big\lvert \mathbb{E}\Big(\big[f\big(\Psi(y)\big)-g\big(\Psi(y)\big)\big]\1_{K\cap W^c}\Big)\Big\rvert &\leq \frac{1}{\sqrt{\alpha}}\int_\mathcal{Y} \big\lvert f(y)-g(y)\big\rvert dy
		< \frac{\epsilon}{18}.
	\end{align}
Applying the triangle inequality, \eqref{eq:CoareaBound2}, and \eqref{eq:CoareaBound3} to \eqref{eq:CoareaBound1} gives
\begin{align}\label{eq:EpsilonBound4}
		\Big\lvert \mathbb{E}\Big(\big[f\big(\Psi(y)\big)-f\big(\Psi(y_*)\big)\big]\1_{K\cap W^c}\Big)\Big\rvert & < \frac{\epsilon}{18} + \frac{\epsilon}{9} + \frac{\epsilon}{18}
		= \frac{\epsilon}{3}
\end{align}
for all $y$ in $B_\delta(y_*)$. Finally, applying the triangle inequality, \eqref{eq:EpsilonBound2}, \eqref{eq:EpsilonBound3}, and \eqref{eq:EpsilonBound4} to \eqref{eq:EpsilonBound1} gives
	\begin{align*}
		\big\lvert Pf(y)-Pf(y_*)\big\rvert & < \epsilon
	\end{align*}
	for all $y$ in $B_\delta(y_*)$. So $Pf\vert_{U_0}$ is continuous and hence $P$ is strong Feller on $U_0$.
\end{proof}


\subsection{Strong Feller and random splitting}\label{sec:feller_splitting}


We return now to a general random splitting associated to a family of analytic vector fields $\mathcal{V}$. In this setting the above results yield the following.

\begin{proposition}\label{prop:split_feller}
If the Lie bracket condition holds at a point $x_*$ in a $d$-dimensional $\mathcal{V}$-orbit $\mathcal{X}$, then for some $\kappa\in\mathbb{N}$ and open neighborhood $U_0$ of $x_*$ the map $t\mapsto D_t\Phi^\kappa_{ht}(x)$ is a submersion for every $x\in U_0$, $h>0$, and almost every\footnote{The Lebesgue-measure zero set $\{t\in\mathbb{R}^{\kappa n}:t\mapsto D_t\Phi^\kappa_{ht}(x)\ \text{is not a submersion}\}$ may depend on $x$ and $h$.} $t \in \mathbb{R}^{\kappa n}_+$. In particular, the transition kernel $P^\kappa_h$ is strong Feller on $U_0$ for every $h>0$ and has transition density $p_{\kappa,h}:U_0\times\mathcal{X}\to \mathbb{R}_{\geq 0}$ on $U_0$ given by
\begin{align}
	p_{\kappa,h}(x,y) &= \int_{\{t:\Phi^\kappa_{ht}(x)=y\}} \frac{\prod_{i=1}^{\kappa n}\rho(t_i)}{\sqrt{\det M(x,ht)}}\ \mathcal{H}^{\kappa n-d}(dt),
\end{align}
for almost every $y\in \mathcal{X}$ with $p_{\kappa,h}(x,y)=0$ otherwise, where $M(x,ht)\coloneqq D_t\Phi^\kappa(x,ht)D_t\Phi^\kappa(x,ht)^*$.
\end{proposition}

\noindent The proof of \Cref{prop:split_feller} uses the following result from \cite{Mityagin}; see also \cite[Lemma 5.22]{kuchment}.

\begin{lemma}\label{lem:Mityagin}
Let $\Omega$ be a connected, open subset of $\mathbb{R}^n$. If $f:\Omega\to\mathbb{R}$ is analytic and not identically $0$, then $f^{-1}(0)$ has Lebesgue measure zero in $\Omega$.
\end{lemma}

\begin{proof}[Proof of \Cref{prop:split_feller}]
By Theorem \ref{thrm:submersion} there exist $\kappa\in\mathbb{N}$ and $t_*\in \mathbb{R}^{\kappa n}_+$ \om{such that $t\mapsto\Phi^\kappa(x_*,t)$ is a submersion at $t_*$.} Define $f:\mathcal{X}\times\mathbb{R}^{\kappa n}_+\to\mathbb{R}$ by
\begin{align*}
	f(x,t) &\coloneqq f_t(x)
		\coloneqq \det D_t\Phi_t^\kappa(x)D_t\Phi_t^\kappa(x)^*
		= \det M(x,t).
\end{align*}
Then $f(x_*,t_*)>0$ and, by continuity, $U_0\coloneqq f_{t_*}^{-1}((0,\infty))$ is an open neighborhood of $x_* \in \mathcal{X}$. Now since the vector fields (and hence their flows) are analytic and analyticity is preserved under addition, multiplication, composition, and differentiation, $t\mapsto f(x,t)$ is analytic for all $x\in \mathcal{X}$. Also, for any $h>0$ and $x\in U_0$ we have $f(x,h(t_*/h))=f(x,t_*)>0$ so $t\mapsto f(x,t)$ is not identically $0$. And since $\mathbb{R}^{\kappa n}_+$ is connected and open in $\mathbb{R}^{\kappa n}$, \Cref{lem:Mityagin} implies $M(x,h\tau)$ is almost-surely invertible. Hence $t\mapsto \Phi^\kappa_{ht}(x)$ is a submersion for every $x\in U_0$, $h>0$, and almost every $t\in \mathbb{R}^{\kappa n}_+$. This proves the first part of the theorem. The expression for the transition density and the strong Feller property of $P^\kappa_h$ on $U_0$ then follow immediately from \Cref{cor:density} and \Cref{thrm:feller}, respectively.
\end{proof}


\section{General Properties Ruling Out Alternatives \ref{alt:conformal} and \ref{alt:subspaces}}\label{sec:alternatives}


If the hypotheses of Theorem \ref{thrm:main} hold and
$d\lambda_1(h)=\lambda_\Sigma(h)$, then Alternative \ref{alt:conformal} or
\ref{alt:subspaces} holds on an open set $U$ in $\mathcal{X}$
satisfying $\mu(U)=1$. In this section we give sufficient conditions
for these alternatives to not hold so that, in particular,
$d\lambda_1(h)\neq\lambda_\Sigma(h)$ whenever the aforementioned  hypotheses
are true. The primary mechanism for ruling out Alternative
\ref{alt:conformal} is shearing (Proposition \ref{prop:conformal});
Alternative \ref{alt:subspaces} is ruled out when the Lie bracket
condition holds at any point in $TU$ (Proposition
\ref{prop:subspaces}). \aa{Note that both these Propositions indirectly use our ongoing assumption that the support of $\rho$ includes the interval $(0,\epsilon)$ for $\epsilon > 0$ small enough as explained in the proof of \Cref{thrm:main}.}

\begin{proposition}\label{prop:conformal}
	Suppose there exist $i,j,k$, and $\ell$ \aa{with $\ell \not \in \{j,k\}$} and a constant $C \neq 0$ such that
	\begin{align}\label{eq:shear}
		V_i(x) & = Cx_\ell(x_k e_j - x_j e_k),
	\end{align}
where $\{e_j\}$ is the standard basis for $\mathbb{R}^D$. Then Alternative \ref{alt:conformal} does not hold.
\end{proposition}

\begin{proof}
	The solution of $\dot{x}=V_i(x)$ starting from $x(0)$ is
	\begin{equ}
		\begin{cases}
			x_j(t) = x_j(0)\cos(C x_\ell t) + x_k(0)\sin(C x_\ell t)    \\
			{x}_k(t) = -x_j(0)\sin(C x_\ell t) + x_k(0)\cos(C x_\ell t) \\
			{x}_p(t) = {x}_p(0), \quad p\not\in\{j,k\}.
		\end{cases}
	\end{equ}
Restricting attention to $x_j,x_k$, and $x_\ell$ since these are the only coordinates that contribute nontrivially to the flow $\varphi\coloneqq\varphi^{(i)}$ of $V_i$, the derivative of $\varphi$ in the $j,k$, and $\ell$ coordinates is
	\begin{equ}
		\aa{D\phi_t(x)} =
		\begin{pmatrix}
			\partial_j\varphi_j    & \partial_k\varphi_j    & \partial_\ell\varphi_j    \\
			\partial_j\varphi_k    & \partial_k\varphi_k    & \partial_\ell\varphi_k    \\
			\partial_j\varphi_\ell & \partial_k\varphi_\ell & \partial_\ell\varphi_\ell
		\end{pmatrix}
		=
		\begin{pmatrix}
			\phantom{-}\cos(C x_\ell t) & \sin(C x_\ell t) & -C t \left(x_j\sin(C x_\ell t) - x_k\cos(C x_\ell t)\right)  \\
			-\sin(C x_\ell t)           & \cos(C x_\ell t) & -C t \left(x_j\cos(C x_\ell t) +  x_k\sin(C x_\ell t)\right) \\
			\phantom{-}0                & 0                & \phantom{-}1
		\end{pmatrix}.
	\end{equ}
	Evaluating at $t_{m'} \coloneqq 2 \pi {m'}/C x_\ell$ for any $x$ with $x_\ell\neq 0$ gives
	\begin{equ}\label{e:solution}
\aa{D\phi_{t_{m'}}(x)} =
		\begin{pmatrix}
			1 & 0 & \phantom{-}2 \pi {m'} \frac{x_k}{x_\ell} \\
			0 & 1 & -2 \pi {m'} \frac{x_j}{x_\ell}           \\
			0 & 0 & \phantom{-}1
		\end{pmatrix}
		= I + {m'}A
		\quad \text{where} \quad
		A =
		\begin{pmatrix}
			0 & 0 & \phantom{-}2 \pi \frac{x_k}{x_\ell} \\
			0 & 0 & -2 \pi \frac{x_j}{x_\ell}           \\
			0 & 0 & 0
		\end{pmatrix}.
	\end{equ}
	Suppose now Alternative \ref{alt:conformal} holds, i.e. there exists a Riemannian structure $\{g_x:x\in U\}$ on an open subset $U$ of $\mathcal{X}$ such that \eqref{eq:conformal} is valid \aa{for all $m$, $t \in \mathbb R_{\geq 0}^{mn}$}, $x\in U$, and $\eta,\xi \in T_x\mathcal{X}$. \aa{In particular, for $t=t_{m'}$ as above and $m=1$, since $\phi_{t_{m'}}(x) = x$ and $\det D \phi_{t_{m'}} (x) = 1$, \eqref{eq:conformal} reads}
	\begin{equ}\label{e:conformal}
		g_{\Phi_{h\tau}(x)}(D\Phi_{h\tau}(x)\eta, D\Phi_{h\tau}(x)\xi) = g_{\varphi_{t_{m'}}(x)}(D\varphi_{t_{m'}}(x)\eta, D\varphi_{t_{m'}}(x)\xi) = g_x(\eta, \xi).
	\end{equ}
	\aa{Hence, choosing $\eta=\xi$ with $A \xi \neq 0$ and  $\|\xi\|_{g_x} = 1$\footnote{{Here, $\|\xi\|_{g_x}$ denotes the norm on $T_x \mathcal X$ induced by $g$}}, by \eref{e:conformal} we must have that for any fixed $x$ with $x_\ell\neq 0$, for all ${m'}$
	\begin{equ}
		g_{x}(D\phi_{t_{m'}}\eta, D\phi_{t_{m'}} \xi) = \sum_j \lambda_j (w_j^\top D\phi_{t_{m'}}(x) \xi)^2 \aa{= 1}.
	\end{equ}
	However, since $\|D\phi_{t_{m'}}\eta\|^2 = \|\xi + {m'} A \xi\|^2 \to \infty$ as ${m'} \to \infty$  we have $g_{x}(D\phi_{t_{m'}}\eta, D\phi_{t_{m'}} \xi) \to \infty$, leading to a contradiction.}
	 So Alternative \ref{alt:conformal} cannot hold.
\end{proof}

\begin{proposition}\label{prop:subspaces}
If the Lie bracket condition holds at $\widetilde{x}_*$ in $TU$, then Alternative \ref{alt:subspaces} cannot hold.
\end{proposition}

\begin{proof}
	Analyticity of the vector fields implies the lifted splitting
	\begin{align}\label{eq:lift2}
		\widetilde\Phi^m_{h\tau}(x,\eta) & \coloneqq \big(\Phi^m_{h\tau}(x), D_x\Phi^m_{h\tau}(x)\eta\big),
	\end{align}
	which we now regard as a chain on $T\mathcal{X}$ rather than $P\mathcal{X}$, is also analytic. Therefore the Lie bracket condition at $\widetilde{x}_*$ together with an argument essentially identical to the proof of \Cref{thrm:feller} gives the existence of an $m$ such that the transition kernel of the lifted process $\widetilde P^m_h$ is strong Feller on a neighborhood $\widetilde{U}$ of $\widetilde{x}_*$ for every $h>0$. Fix such an $h$ and assume for simplicity $m=1$, which comes without loss of generality since the Lyapunov exponents of $\{\Phi^m_{h\tau}\}_{m=0}^\infty$ are \aa{a multiple\footnote{namely
$
		\lambda_k^{(\kappa)}(h) \coloneqq \lim_{m\to\infty}\tfrac{1}{m}\log\lVert D_x\Phi^{\kappa m}_{h\tau}(x)\eta\rVert
			= \kappa\lim_{m\to\infty}\tfrac{1}{\kappa m}\log \lVert D_x\Phi^{\kappa m}_{h\tau}(x)\eta\rVert
			= \kappa\lambda_k(h).
$
	That is, the Lyapunov exponents $\lambda^{(\kappa)}_k$ of $\{\Phi^{\kappa m}_{h\tau}\}$ are $\kappa$ times the Lyapunov exponents of $\{\Phi^m_{h\tau}\}$.} of} those of $\{\Phi^{m\kappa }_{h\tau}\}_{m=0}^\infty$  and both alternatives in Theorem \ref{thrm:main} hold for all $m$. In particular, $\lambda_1(h)>0$ if and only if $\lambda^{(\kappa)}_1(h)>0$. Also $d\lambda_1(h)=\lambda_\Sigma(h)$ if and only if $d\lambda_1^{(\kappa)}(h)=\lambda^{(\kappa)}_\Sigma(h)$. Also, by shrinking $\widetilde{U}$ if necessary, assume the projection $\pi(\widetilde{U})$ of $\widetilde{U}$ onto $\mathcal{X}$ is contained in $U$. If Alternative \ref{alt:subspaces} holds, by definition there exist for every $x\in U$ proper linear subspaces $E^1_x,\dots,E^p_x$ of $T_x\mathcal{X}$ such that
	\begin{align}\label{eq:invariant}
		D_x\Phi^\kappa_t(x)(E^i_x) & = E^{\sigma(i)}_{\Phi^\kappa_t(x)}
	\end{align}
	for \aa{every $i$ and every $t \in \mathbb R_{\geq 0}^{\kappa n}$,}
	 where $\sigma$ is a permutation of $[p]$. In particular, setting
	\begin{align*}
		E_x & \coloneqq \bigcup_{i=1}^p E_x^i
	\end{align*}
	for each $x\in U$, the map $f:\widetilde{U}\to\mathbb{R}$
        given by\footnote{ The assumption that $\pi(\widetilde{U})$ is contained in $U$ guarantees $f$ is well-defined.}
	\begin{align*}
	f(x,\eta) & \coloneqq \1_{E_x}(\eta)
		\coloneqq
		\begin{cases}
			1 & \text{if}\ \eta\in E_x, \\
			0 & \text{otherwise},
		\end{cases}
	\end{align*}
	is bounded, measurable, and -- since the $E^i_x$ are \textit{proper} subspaces -- discontinuous. But by \eqref{eq:invariant},
	\begin{align*}
		\widetilde P^\kappa_hf(x,\eta) & = \mathbb{E}\big(\1_{E_{\Phi^\kappa_{h\tau}(x)}}\left(D_x\Phi^\kappa_{h\tau}(x)\eta\right)\big)
		= \mathbb{E}\left(\1_{D_x\Phi^\kappa_{h\tau}(x)(E_x)}\left(D_x\Phi^\kappa_{h\tau}(x)\eta\right)\right)
		= f(x,\eta)
	\end{align*}
	is discontinuous, which contradicts that $\widetilde P^\kappa_h$ is strong Feller on $\widetilde{U}$. So Alternative \ref{alt:subspaces} cannot hold.
\end{proof}

\begin{remark}
\om{The Lie bracket condition in \Cref{prop:subspaces} applies to the lifted vector fields $\widetilde V_j(x,\eta)=(V_j(x),DV_j(x)\eta)$. As mentioned above, it is harder in general to verify the Lie bracket condition for the lifted family $\widetilde V_j$ than the $V_j$. It is also worth noting the Lie bracket condition cannot hold on the zero section of $TU$ since $DV_j(x)0 = 0$ for all $x$ and $j$.}
\end{remark}


\section{Positive top Lyapunov exponent: Conservative Lorenz-96 and 2d Euler}\label{sec:models2}


We know from \Cref{thrm:models_ergodic} that for any generic orbit $\mathcal{X}$ of the Lorenz and Euler splittings there exists a unique
$P_h$-invariant measure $\mu$ on $\mathcal X$ for every $h>0$. Thus the Lyapunov exponents exist and are almost-surely constant on generic orbits. We also noted in \Cref{rmk:invariant} that each such $\mu$ is the disintegration of Lebesgue measure onto its respective orbit and is therefore invariant under $\Phi_{h\tau}$ for every $\tau$. In particular, the pushforward measures $\mu_m$ defined in \Cref{sec:main} satisfy $\mu_m=\mu$ and hence $\mathbb{E}D_{KL}(\mu_m\parallel\mu)=0<\infty$ for all $m$. Furthermore $\lambda_\Sigma(h)=0$ (and hence $\lambda_1(h)\geq 0$) since the splitting vector fields conserve Euclidean norm. This establishes the hypotheses of Theorem \ref{thrm:main} for random splittings on generic orbits of both Lorenz and Euler. And since $\lambda_\Sigma(h)=0$ in both cases, Theorem \ref{thrm:main} says that if $\lambda_1(h)=0$ then Alternative \ref{alt:conformal} or \ref{alt:subspaces} must hold. So to prove Theorems \ref{thrm:lorenz} and \ref{thrm:euler}, it remains to show neither alternative holds in both models. In Sections \ref{sec:model_alt1}  and \ref{sec:model_alt2} we rule out Alternatives \ref{alt:conformal} and \ref{alt:subspaces}, respectively, for both models. Since the latter is more involved, we separate \Cref{sec:model_alt2} into two subsections: \Cref{sec:lorenz_alt2} for Lorenz and \Cref{sec:euler_alt2} for Euler.


\subsection{Ruling out Alternative \ref{alt:conformal}}\label{sec:model_alt1}


Recall the splitting vector fields for Lorenz are
\begin{align*}
	V_j(x) &= x_{j-1}(x_{j+1}e_j-x_je_{j+1}),
\end{align*}
so Alternative \ref{alt:conformal} is immediately ruled out by Proposition \ref{prop:conformal}. For Euler, simple computation shows the $C_{jk}$ defined in \eqref{eq:constants} satisfy $C_{jk}=0$ and $C_{j\ell}=-C_{k\ell}$ whenever $\lvert j\rvert=\lvert k\rvert$ and $j+k+\ell=0$, e.g. when $j=(1,0)$, $k=(0,1)$, and $\ell=-(1,1)$. In this case the equation $\dot{q}=V_{a_ja_ka_\ell}(q)$ is given by
\begin{align*}
	\begin{cases}
		\dot{a}_j = C_{j\ell}a_\ell a_k \\
		\dot{a}_k = -C_{j\ell}a_\ell a_j \\
		\dot{a}_\ell = 0,
	\end{cases}
\end{align*}
which is equivalent to \eqref{eq:shear}. So Proposition \ref{prop:conformal} rules out Alternative \ref{alt:conformal} for Euler as well. All that remains then is to rule out Alternative \ref{alt:subspaces}, which we do via Proposition \ref{prop:subspaces}.

\begin{remark}
  For simplicity and efficiency, we have only leveraged the shearing
  in the diagonal case when $\lvert j\rvert=\lvert k\rvert$. Shearing,
  sufficient to rule out  Alternative \ref{alt:conformal},
  also exists in other non-diagonal triads.
\end{remark}


\subsection{Ruling out Alternative \ref{alt:subspaces}}\label{sec:model_alt2}


The splitting vector fields $V_j$ of a general random splitting on a $\mathcal V$-orbit $\mathcal{X}$ lift to vector fields $\widetilde{V}_j$ on the tangent bundle $T\mathcal{X}$ given by
\begin{align}\label{eq:lifted_fields}
	\widetilde{V}_j(\widetilde{x}) &=
		\begin{pmatrix}
			V_j(x) \\
			DV_j(x)\eta
		\end{pmatrix},
\end{align}
where $\widetilde{x}=(x,\eta)$. When regarded as vector fields on the ambient space $\mathbb{R}^D$, which is the perspective we take in what follows, the Lie brackets of the lifted vector fields are given by
\begin{align*}
	\big[\widetilde{V}_i,\widetilde{V}_j\big](\widetilde{x}) &= D\widetilde{V}_j(\widetilde{x})\widetilde{V}_i(\widetilde{x}) - D\widetilde{V}_i(\widetilde{x})\widetilde{V}_j(\widetilde{x}),
\end{align*}
where the derivative $D$ is with respect to $\widetilde{x}$. As noted above when ruling out Alternative \ref{alt:conformal}, for both Lorenz and Euler each of the splitting vector fields effectively depends on only three coordinates. We therefore reorder the coordinates $(x_1,\dots,x_D,\eta_1,\dots,\eta_D)$ as $(x_1,\eta_1,\dots,x_D,\eta_D)$ in what follows.


\subsubsection{Conservative Lorenz-96}\label{sec:lorenz_alt2}


Fix $n\geq 4$. The generic orbits of the Lorenz splitting are
\begin{align}\label{eq:lorenz_orbit}
	\mathcal{X}_R &= \bigg\{x\in\mathbb{R}^n : \lVert x\rVert=R\ \text{and}\ \sum_{j=1}^n(x_j^2+x_{j+1}^2)x_{j-1}^2\neq 0\bigg\},
\end{align}
for $R>0$. Each $\mathcal{X}_R$ is precisely the points on the sphere
$\mathbb{S}^{n-1}_R$ of radius $R$ in $\mathbb{R}^n$ that are not
fixed by \textit{all} the $V_j$. In particular, $\mathcal{X}_R$ is a
codimension 1 submanifold of $\mathbb{R}^n$ satisfying
$T_x\mathcal{X}_R=T_x\mathbb{S}^{n-1}_R$ for all $x\in \mathcal{X}_R$.
See \cite{Agazzi} for more discussion.
Fixing $R>0$, for brevity we  let $\mathcal{X}$ denote the
$\mathcal{X}_R$ from \eqref{eq:lorenz_orbit}. Our objective is to find a point $\widetilde{x}_* \in T\mathcal{X}$ at which the family $\widetilde{\mathcal{V}}=\{\widetilde{V}_j\}$ satisfies the Lie bracket condition, $\dim(\Lie_{\widetilde{x}}(\widetilde{\mathcal{V}}))=2n-2$.

In the coordinates $(x_1,\eta_1,\dots,x_n,\eta_n)$, the lifted vector fields of the Lorenz splitting are
\begin{align}\label{eq:lorenz_lift}
	\widetilde{V}_i(x,\eta) &=
		(
			0,
			\dots,
			0,
			x_{i-1}x_{i+1},
			\eta_{i-1}x_{i+1}+\eta_{i+1}x_{i-1},
			-x_{i-1}x_i,
			-\eta_{i-1}x_i-\eta_ix_{i-1},
			0,
			\dots,
			0
		),
\end{align}
where, in order from left to right, the nonzero entries correspond to the coordinates $x_i,\eta_i,x_{i+1}$, and $\eta_{i+1}$. Let $\widetilde{x}=(x,\eta)$ be any point of $T\mathcal{X}$ satisfying
\begin{align}\label{eq:lorenz_point}
	x &= (a,a,b,b,b,\dots,b)
		\quad\text{and}\quad
		\eta =
			\begin{cases}
				(1,-1,1,-1,\dots,1,-1), & \text{if $n$ even} \\
				(1,-1,1,-1,\dots,1,-1,0), & \text{if $n$ odd},
			\end{cases}
\end{align}
with $a,b\neq 0$. Note $\eta$ is perpendicular to $x$ as elements of $\mathbb{R}^n$ and is therefore a well-defined element of $T_x\mathcal{X}=T_x\mathbb{S}^{n-1}(R)$. Consider first the case when $n$ is even. Direct computation via \eqref{eq:lorenz_lift} shows that for $i=2,\dots,n-1$ the vector fields $\widetilde{V}_i$ and $[\widetilde{V}_{i-1},\widetilde{V}_i]$ evaluated at $\widetilde{x}$ form the $2n\times 2$ matrix
\begin{align*}
	\begin{pmatrix}
		\vline	&	\vline \\
		\big[\widetilde{V}_{i-1},\widetilde{V}_i\big](x,\eta) & \widetilde{V}_i(x,\eta)  \\
		\vline	&	\vline
	\end{pmatrix}
		&=
		\begin{pmatrix}
			\bigstar \\
			A_i \\
			\text{\large \textbf{0}}
		\end{pmatrix},
\end{align*}
where $\bigstar$ indicates irrelevant entries, \textbf{0} indicates the rest of the matrix is filled with zeros, and the $2\times 2$ matrix $A_i$, which comprises the $2i+1$ and $2i+2$ rows of the matrix, is given by
\begin{align*}
	A_2 &=
		\begin{pmatrix}
			0 & -a^2 \\
			4ab & \ph 0
		\end{pmatrix},
		\quad
		A_3 =
			\begin{pmatrix}
				\ph a(a^2-b^2) & -ab \\
				-(a+b)^2 & \ph b-a
			\end{pmatrix},
		\quad
		A_4 =
			\begin{pmatrix}
				0 & -b^2 \\
				4ab & \ph 0
			\end{pmatrix},
		\quad
		A_i =
			\begin{pmatrix}
				\ph 0 & -b^2 \\
				\pm 4b^2 & \ph 0
			\end{pmatrix},
\end{align*}
with the last $A_i$ holding for all $i>4$. Define the $2n\times 2n-2$ matrix
\begin{align*}
	\mathcal{A} &\coloneqq
	{\small
	\begin{pmatrix}
		\vline & \vline & \vline & & \vline & \vline & \vline \\
		\widetilde{V}_1 & [\widetilde{V}_1,\widetilde{V}_2] & \widetilde{V}_2 & \cdots & [\widetilde{V}_{n-2},\widetilde{V}_{n-1}] & \widetilde{V}_{n-1} & \widetilde{V}_n  \\
		\vline & \vline & \vline & & \vline & \vline & \vline
	\end{pmatrix}
	}
	=
	{\small
	\begin{pmatrix}
		\\
		\multicolumn{8}{c}{B} \\
		\\
		\hline
		0& \multicolumn{1}{|c|}{A_2} &\star&\star& \star&\star&\star&\star \\
		\cline{2-3}
		0&0& \multicolumn{1}{|c|}{A_3} &\star&\star&\star&\star&\star \\
		\cline{3-4}
		0&0&0& \multicolumn{1}{|c|}{A_4} & \star & \star & \star&\star \\
		\cline{4-5}
		0&0&0&0& \multicolumn{1}{|c|}{A_5} & \star & \star & \star \\
		\cline{5-5}
		\vdots&\vdots&\vdots&\vdots&\vdots & \ddots & \vdots &\vdots \\
		\cline{7-7}
		0&0&0&0&0&\cdots & \multicolumn{1}{|c|}{A_{n-1}} & \star \\
		\cline{7-7}
	\end{pmatrix},
	}
\end{align*}
where $B$ is the $4\times 2n-2$ matrix
\begin{align*}
	B &=
		\begin{pmatrix}
			\ph ab & -ab^2 & 0 &\ph 0 &0&\cdots&0&0 \\
			-a-b & -b^2 & 0 & \ph 0 &0&\cdots&0&4b^2 \\
			-ab & \ph ab^2 & ab & -ab^2 &0&\cdots&0&0 \\
			\ph a-b & -b^2 & a+b & \ph a^2 &0&\cdots&0&0
		\end{pmatrix}.
\end{align*}
We claim $\mathcal{A}$ has rank $2n-2$ for certain choices of $a$ and $b$. First, note that $A_2,A_4$, and $A_i$, $i>4$, have rank $2$ whenever $a,b\neq 0$. Also $A_3$ has rank $2$ whenever $a,b\neq 0$ and
\begin{align}\label{eq:a3}
	a^3+ab^2+2b^3 &\neq 0.
\end{align}
Row reducing $B$ gives the matrix
\begin{align*}
	B' &=
		\begin{pmatrix}
			ab & \star & 0 & \ph 0 &0&\cdots&0&0 \\
			0 & \star & 0 & \ph 0 &0&\cdots&0&\star \\
			0 & 0 & \star & \ph \star &0&\cdots&0&0 \\
			0 & 0 & 0 & -a^5b^5(2a^2+5ab+2b^2) &0&\cdots&0&4a^4b^7(2b-a) \\
		\end{pmatrix}.
\end{align*}
The $\star$ entries, though easily computed and simply expressed, are redacted to emphasize the relevant terms. Suppose that, in addition to $a,b\neq 0$, the relations
\begin{align}\label{eq:relations}
	\begin{cases}
		2a^2 + (n-2)b^2 = R^2    \\
		2a^2 + 5ab + 2b^2 = 0    \\
		a^3 + ab^2 + 2b^3 \neq 0 \\
		2b - a \neq 0
	\end{cases}
\end{align}
hold. Direct substitution verifies all the above are satisfied when
\begin{align}\label{eq:ab}
	a & = -\frac{R}{\sqrt{4n-6}}
	\quad\text{and}\quad
	b = \frac{\sqrt{2}R}{\sqrt{2n-3}}.
\end{align}
The first relation in \eqref{eq:relations} guarantees $x=(a,a,b,\dots,b)$ satisfies $\lVert x\rVert=R$ and therefore lies on $\mathcal{X}$, and the third guarantees $A_3$ has rank 2 by \eqref{eq:a3}. The second and fourth guarantee $B'$ has the form
\begin{align*}
	B' & =
	\begin{pmatrix}
		ab & \star & 0     & 0     & 0 & \cdots & 0 & 0     \\
		0  & \star & 0     & 0     & 0 & \cdots & 0 & \star \\
		0  & 0     & \star & \star & 0 & \cdots & 0 & 0     \\
		0  & 0     & 0     & 0     & 0 & \cdots & 0 & c     \\
	\end{pmatrix}
\end{align*}
for some nonzero constant $c$. Replacing $B$ with $B'$ in $\mathcal{A}$ and moving the fourth row of $B'$ to the last row of the whole matrix gives a new matrix
\begin{align}\label{eq:new_matrix}
	{\small
		\begin{pmatrix}
			\multicolumn{1}{c|}{ab} & \multicolumn{7}{c}{}                                                                                                                                                               \\
			\cline{1-1}
			\multicolumn{1}{c|}{0}  & \multicolumn{7}{c}{\bigstar}                                                                                                                                                       \\
			\multicolumn{1}{c|}{0}  & \multicolumn{7}{c}{}                                                                                                                                                               \\
			\cline{2-8}
			0                       & \multicolumn{1}{|c|}{A_2}    & \star                     & \star                     & \star                     & \cdots & \star                         & \star                  \\
			\cline{2-3}
			0                       & 0                            & \multicolumn{1}{|c|}{A_3} & \star                     & \star                     & \cdots & \star                         & \star                  \\
			\cline{3-4}
			0                       & 0                            & 0                         & \multicolumn{1}{|c|}{A_4} & \star                     & \cdots & \star                         & \star                  \\
			\cline{4-5}
			0                       & 0                            & 0                         & 0                         & \multicolumn{1}{|c|}{A_5} & \cdots & \star                         & \star                  \\
			\cline{5-5}
			\vdots                  & \vdots                       & \vdots                    & \vdots                    & \vdots                    & \ddots & \vdots                        & \vdots                 \\
			\cline{7-7}
			0                       & 0                            & 0                         & 0                         & 0                         & \cdots & \multicolumn{1}{|c|}{A_{n-1}} & \star                  \\
			\cline{7-8}
			0                       & 0                            & 0                         & 0                         & 0                         & \cdots & 0                             & \multicolumn{1}{|c}{c}
		\end{pmatrix}.
	}
\end{align}
Since each $A_i$ has rank 2 and $ab,c\neq 0$, this matrix -- and hence $\mathcal{A}$ itself -- has rank $2n-2$.

When $n$ is odd, everything is essentially the same. Only the matrix $B$ is different, but its row reduced form is identical to the $B'$ above, up to the irrelevant $\star$ terms. In particular, the matrix corresponding to $\mathcal{A}$ in the odd case becomes the matrix \eqref{eq:new_matrix} via the exact same procedure detailed above when subjected to the same relations defined in \eqref{eq:relations}. Thus the the Lie bracket condition holds for the lifted process at the point $\widetilde{x}$ defined as in \eqref{eq:lorenz_point} with $a$ and $b$ as in \eqref{eq:ab} for all $n\geq 4$. So by Proposition \ref{prop:subspaces} Alternative \ref{alt:subspaces} cannot hold. And by Theorem \ref{thrm:main} the top Lyapunov exponent of the random splitting of conservative Lorenz-96 is positive.


\subsubsection{2D Euler}\label{sec:euler_alt2}


As in the previous section, the main idea of the proof is to fix a point $(q^*, \eta^*)\in T\mathcal X$, check that the Lie Bracket condition holds for $\mathcal V$ from \eqref{eq:euler_fields} at that point, and invoke \Cref{prop:subspaces}. In this spirit, we proceed to choose a subset of vector fields and related commutators (indexed by the related triples of interacting indices) from $\mathcal{V}$ whose spanning dimension, when  evaluated at $(q^*, \eta^*)$, can be readily deduced. Concretely, this will be done by computing the rank of the matrix whose columns are given by such vector fields. To increase readibility, while presenting the full idea of the computation and its results we have suppressed the lengthy algebraic manipulations such computation entails. In the interest of reproducibility, Mathematica code to reproduce such computations is available at \cite{code}. The code's inputs and outputs are reported in Appendix~\ref{a:computations}.

Focusing on a given triple $j,k,\ell \in \mathbb{Z}^2_N$, we define the vector fields in $\mathcal X$ which we express in $6$ dimensions -- corresponding to coordinates $(a_j,b_j, a_k,b_k,a_\ell,  b_\ell)$ as
\begin{equ}\label{e:vflong}
	V_{jk\ell}^{(1)} =
	\begin{pmatrix}
		-\ckl a_k a_\ell \\
		0                \\
		-\clj a_j a_\ell \\
		0                \\
		-\cjk a_ja_k     \\
		0
	\end{pmatrix}
	\quad
	V_{jk\ell}^{(2)} =
	\begin{pmatrix}
		\ckl b_kb_\ell \\
		0              \\
		0              \\
		\clj a_jb_\ell \\
		0              \\
		\cjk a_jb_k
	\end{pmatrix}
	\quad
	V_{jk\ell}^{(3)} =
	\begin{pmatrix}
		0              \\
		\ckl a_kb_\ell \\
		\clj b_jb_\ell \\
		0              \\
		0              \\
		\cjk b_ja_k
	\end{pmatrix}
	\quad
	V_{jk\ell}^{(4)} =
	\begin{pmatrix}
		0              \\
		\ckl b_ka_\ell \\
		0              \\
		\clj b_ja_\ell \\
			\cjk b_jb_k    \\
		0
	\end{pmatrix}
\end{equ}
The commutators of these vector fields read
\begin{align*}\tiny
	[V_{jk\ell}^{(1)},V_{jk\ell}^{(2)}]  =
	\begin{pmatrix}
		0                         \\
		0                        \\
		\ckl\clj a_\ell b_k b_\ell  \\
		-	\ckl\clj b_\ell a_k a_\ell \\
		\ckl\ckj a_k b_k b_\ell     \\
		-\ckl\ckj b_k a_k a_\ell
	\end{pmatrix}
 \qquad
	[V_{jk\ell}^{(1)},V_{jk\ell}^{(3)}]  =
	\begin{pmatrix}
		\ckl\cjl a_\ell b_j b_\ell  \\
		-	\ckl\cjl b_\ell a_j a_\ell \\
		0                         \\
		0                       \\
		\ckj\clj a_j b_j b_\ell     \\
		-\cjl\ckj b_k a_j a_\ell
	\end{pmatrix}
	\qquad
	[V_{jk\ell}^{(1)},V_{jk\ell}^{(4)}]  =
	\begin{pmatrix}
		\ckl\cjk a_k b_j b_k  \\
		-	\ckl\cjk a_j a_k b_k \\
		\cjl\cjk a_j b_j b_k  \\
		-\cjl\cjk a_j a_k b_j \\
		0                   \\
		0
	\end{pmatrix}
\\\tiny
	[V_{jk\ell}^{(2)},V_{jk\ell}^{(3)}]
	                   =
	\begin{pmatrix}
		-	 \ckl\cjk a_k b_j b_k \\
		\ckl\cjk a_j a_k b_k   \\
		\cjl\cjk a_j b_j b_k   \\
		-\cjl\cjk a_j a_k b_j  \\
		0                    \\
		0
	\end{pmatrix}
	\qquad
	[V_{jk\ell}^{(2)},V_{jk\ell}^{(4)}]
	                   =
	\begin{pmatrix}
		-\ckl\cjl a_\ell b_j b_\ell \\
		\ckl\cjl b_\ell a_j a_\ell  \\
		0                         \\
		0                        \\
		\ckj\clj a_j b_j b_\ell     \\
		-\cjl\ckj b_k a_j a_\ell
	\end{pmatrix}\qquad
	[V_{jk\ell}^{(3)},V_{jk\ell}^{(4)}]
	                   =
	\begin{pmatrix}
		0                         \\
		0                       \\
		-\ckl\clj a_\ell b_k b_\ell \\
		\ckl\clj b_\ell a_k a_\ell  \\
		\ckl\ckj a_k b_k b_\ell     \\
		-\ckl\ckj b_k a_k a_\ell
	\end{pmatrix}
\end{align*}
This yields the extended, $12$-dimensional vector fields on $T \mathcal X$ obtained by stacking the vector fields \eqref{e:vflong} and their commutators with the corresponding linearizations. Throughout, we represent such vector fields reordering the coordinates of $T\mathcal X$ for a triple of indices $j,k,\ell \in \mathbb Z_N^2$ as
\begin{equ}
(a_j,b_j,\eta_j^a,\eta_j^b,a_k,b_k,\eta_k^a, \eta_k^b,a_\ell,b_\ell, \eta_\ell^a, \eta_\ell^b)\,.
\end{equ}
where for any $j \in \mathbb Z_N^2$, $\eta_j^a, \eta_j^b$ are the linearization coordinates corresponding to $a_j,b_j$\,.
Then, denoting by $[f]$ the linearization of a map $f~:~\mathbb R^6 \to \mathbb R$  , {i.e.,}
$	[f](x,\eta) := Df \cdot \eta$\,,
we write such extended vector fields for a given triple $j,k,\ell \in \mathbb Z_N^2$ as the columns of a matrix $M_{jk\ell}$, obtaining
\begin{equs}M_{jk\ell} &= \tiny \begin{pmatrix}
	\vline & \vline &\vline	&	\vline&\vline & \vline &\vline	&	\vline&\vline & \vline  \\
	{V}_{jk\ell}^{(1)}& {V}_{jk\ell}^{(2)}&{V}_{jk\ell}^{(3)}& {V}_{jk\ell}^{(4)}&    \big[{V}_{jk\ell}^{(1)},{V}_{jk\ell}^{(2)}\big] & \big[{V}_{jk\ell}^{(1)},{V}_{jk\ell}^{(3)}\big]&\big[{V}_{jk\ell}^{(1)},{V}_{jk\ell}^{(4)}\big] & \big[{V}_{jk\ell}^{(2)},{V}_{jk\ell}^{(3)}\big]&\big[{V}_{jk\ell}^{(2)},{V}_{jk\ell}^{(4)}\big] & \big[{V}_{jk\ell}^{(3)},{V}_{jk\ell}^{(4)}\big] \\
	\vline & \vline &\vline	&	\vline&\vline & \vline &\vline	&	\vline&\vline & \vline
\end{pmatrix} \\
&	={\tiny\begin{pmatrix}
		\ckl a_k a_\ell   & \ckl b_k b_l      & 0                 & 0                 & \bkal     & \blak     & 0        & 0        & -\blak   & - \bkal          \\
		0                  & 0                 & \ckl a_k b_\ell   & \ckl b_k a_\ell   & -\bkbl    & -\blbk    & 0        & 0        & \blbk    & \bkbl           \\
		\ckl [a_k a_\ell] & \ckl[ b_k b_l]    & 0                 & 0                 & [\bkal]   & [\blak]   & 0        & 0        & -[\blak] & - [\bkal]        \\
		0                  & 0                 & \ckl [a_k b_\ell] & \ckl [b_k a_\ell] & -[\bkbl]  & -[\blbk ] & 0        & 0        & [\blbk]  & [ \bkbl ]       \\
				\cjl a_j a_\ell   & 0                 & \cjl b_j b_\ell   & 0                 & \bjal     & 0         & \blaj    & -\blaj   & 0        & \bjal            \\
		0                  & \cjl a_j b_\ell   & 0                 & \cjl b_j a_\ell   & - \bjbl   & 0         & -\blbj   & \blbj    & 0        & -\bjbl           \\
		\cjl [a_j a_\ell] & 0                 & \cjl [b_j b_\ell] & 0                 & [\bjal]   & 0         & [\blaj]  & -[\blaj] & 0        & [\bjal]         \\
0                  & \cjl [a_j b_\ell] & 0                 & \cjl [b_j a_\ell] & - [\bjbl] & 0         & -[\blbj] & [\blbj]  & 0        & -[\bjbl]         \\
				\cjk a_j a_k      & 0                 & 0                 & \ckj b_k b_j      & 0         & \bjak     & \bkaj    & \bkaj    & \bjak    & 0          \\
		0                  & \cjk a_j b_k      & \cjk a_k b_j      & 0                 & 0         & -\bjbk    & -\bkbj   & -\bkbj   & -\bjbk   & 0          \\
		\cjk [a_j a_k]    & 0                 & 0                 & \ckj [b_k b_j]    & 0         & [\bjak]   & [\bkaj ] & [\bkaj]  & [\bjak]  & 0          \\
0                  & \cjk [a_j b_k]    & \cjk [a_k b_j]    & 0                 & 0         & -[\bjbk]  & -[\bkbj] & -[\bkbj] & -[\bjbk] & 0
\end{pmatrix}}\,,
\end{equs}
with
\begin{equ}
	B_{k,a}^\ell := B_k a_\ell \qquad \text{ for }\qquad B_k := \cjk\ckl a_kb_k\,.
	\end{equ}

	We further denote by $M_{jk\ell}'$ the last four rows intersected with the first,  second,  sixth and  seventh column of the above matrix:
	\begin{equ}\label{e:matrix1}\small
	M_{jk\ell}' := \begin{pmatrix}
				\cjk a_j a_k      & 0                      & \bjak     & \bkaj          \\
		0                  & \cjk a_j b_k            & -\bjbk    & -\bkbj         \\
		\cjk [a_j a_k]    & 0                     & [\bjak]   & [\bkaj ]        \\
0                  & \cjk [a_j b_k]          & -[\bjbk]  & -[\bkbj]
	\end{pmatrix}\,.
\end{equ}

To check spanning, we now write a subset of the vector fields introduced above as the columns of a matrix whose rank will be shown to be $\text{dim}T\mathcal X = 4n-4$. To choose such vector fields, we introduce a convenient enumeration of the index space $\mathbb{Z}^2_N$ as the bijection $F:\{0,\dots,2N(N+1)\} \to (\mathbb{Z}^2_N)^3$ given by
\begin{align*}
 F(i) &\coloneqq
	 \begin{cases}
		 (1,0),(-N+i,1),(-N+i+1,1) & i \leq N-2\\
		 (-1,1),(-3,1),(2,0) & i=N-1\,, \\
		 (2,0),(-2,1),(0,1) & i=N\,\\
		 (0,1),(2,0),(2,1) & i = N+1\\
		 (1,0),(2,1),(1,1) & i = N+2\\
		 (1,0),(-N+i-1,1),(-N+i,1) & N+3\leq i \leq 2N\\
		 (0,1),(-2N+i,1),(-2N+i+1,0) & 2N+1\leq i \leq 3N-1\,\\
		 (1,0),(L_1(i),L_2(i)),(L_1(i)+1,L_2(i)) & 3N\leq i \leq 2N(N+1)\,
	 \end{cases}
\end{align*}
where $(L_1(i),L_2(i))$ is the element of $\mathbb{Z}^2_N$ obtained by starting from $(-N,1)$ (for $j = 3N$) and for each new value of $j$ proceeding incrementally to the right (i.e., adding $1$ to the first component of the two-dimensional index) until $(N,1)$, then moving to the row above at $(-N,2)$, then again moving progressively to the right until $(N,2)$ and so on until $(N,N)$. Formally $(L_1(i), L_2(i))$ is therefore defined as
\begin{equs}
	L_1(i) &:= -N+\text{mod}(i-3N,2N+1) \\
	L_2(i) &:= \lfloor (i-3N)/(2N+1)\rfloor\,,
\end{equs}
where $\text{mod}(\,\cdot\,,\,\cdot\,)$ and $\lfloor  \,\cdot\,\rfloor$ respectively denote integer part and the modulo operation, i.e., the remainder of division of the first argument by the second.
Note that this enumeration differs from the one exemplified in \cite[Figure 2]{Agazzi}. This ordering of the indices is motivated by the fact that each time we let a new triple interact, two of the indices have already interacted in a previous triple, adding exactly a new one.  Because of this, from this enumeration of the triples follows an enumeration of the indices: after mapping the first three interacting indices $(1,0),(-N+i,1),(-N+i+1,1)$ to $1,2,3$, we proceed incrementally assigning the number $i$ to the third element of $F(i-3)$. We also note for future reference that for all interacting triples considered above we have nonvanishing interaction constants $C_{jk}, C_{j\ell}, C_{k\ell}$.

Using the above enumeration we can define
\begin{align*}
 \mathcal A_i = \begin{pmatrix}
	 \vline & \vline &\vline	&	\vline \\
	 {V}_{F(i)}^{(1)}& {V}_{F(i)}^{(2)}&    \big[{V}_{F(i)}^{(1)},{V}_{F(i)}^{(3)}\big] & \big[{V}_{F(i)}^{(1)},{V}_{F(i)}^{(4)}\big]  \\
	 \vline	&	\vline& \vline &\vline
 \end{pmatrix}
	 &=
	 \begin{pmatrix}
		 \bigstar \\
		 M_{F(i)}' \\
		 \text{\large \textbf{0}}
	 \end{pmatrix},
\end{align*}
where $\bigstar$ indicates irrelevant entries, \textbf{0} indicates the rest of the matrix is filled with zeros. We then define the $4n\times 4n-4$ matrix
\begin{align*}
 \mathcal{A} &\coloneqq
 {\small
 \begin{pmatrix}
	 M_{0} & \vline & \vline & & \vline \\
		&\mathcal A_1  & \mathcal A_1 & \dots&\mathcal A_{n-3}  \\
	 \text{\large \textbf{0}} & \vline & \vline & & \vline
 \end{pmatrix}
 }
 =
 {\small
 \begin{pmatrix}
	 &&\multicolumn{1}{c|}{ }   &\star&\star&\star& \star&\star&\star\\
	 &\multicolumn{1}{c}{ M_{0}}&\multicolumn{1}{c|}{ }  &\star&\star&\star& \star&\star&\star\\
	 &&\multicolumn{1}{c|}{ }   &\star&\star&\star& \star&\star&\star\\
	 \cline{1-4}
	 &0&& \multicolumn{1}{|c|}{M_{F(1)}'} &\star&\star& \star&\star&\star \\
	 \cline{4-5}
	 &0&&0& \multicolumn{1}{|c|}{M_{F(2)}'} &\star&\star&\star&\star \\
	 \cline{5-6}
	 &0&&0&0& \multicolumn{1}{|c|}{M_{F(3)}'} & \star & \star & \star \\
	 \cline{6-7}
	 &0&&0&0&0& \multicolumn{1}{|c|}{M_{F(4)}'} & \star & \star &  \\
	 \cline{7-7}
	 &\vdots&&\vdots&\vdots&\vdots&\vdots & \ddots & \vdots  \\
	 \cline{9-9}
	 &0&&0&0&0&0&\cdots & \multicolumn{1}{|c|}{M_{F(n-3)}'}  \\
	 \cline{9-9}
 \end{pmatrix},
 }
\end{align*}
where $M_{0}$ is a $12\times 8$ matrix obtained by removing the eigth and the tenth column from $M_{F(0)}$.
Note that the above matrix inherits its structure from the indexing of
the interacting triples: each triple contains exactly one index that
has not yet interacted with other modes. To prove the desired result,
it therefore remains to show that for every compatible\footnote{ We call a pair of conserved quantities compatible when there exists a nondegenerate point in state space with the given energy and enstrophy. This is true if $E/(2N^2) < \mathcal E < E$} choice of conserved quantities $(E,\mathcal E)$ there exists a point $(q^*, \eta^*)$ where each of the $n-3$ block elements has rank $4$ and the $M_{0}$ matrix has rank $8$, thereby yielding the desired total rank of $4n-4$.
Proceeding with this plan, in the following we fix the vector $(q^*,\eta^*)$ where we establish spanning as
\begin{equ}
	a_j^*  = \begin{cases}
	\alpha_1\quad &\text{if } j \in \{(0,1),(1,0)\}\\
		\alpha_2\quad &\text{if } j = (N,N)\\
			\beta \quad &\text{else}
	\end{cases}
	\qquad
	b_j^*  = \begin{cases}
	2a_j^*\quad &\text{if } j \in \{(0,1),(1,0),(-1,1),(2,0)\}\\
		a_j^*\quad &\text{else }
	\end{cases}
\end{equ}
\begin{equ}
	(\eta^*)_j^\#  = \begin{cases}
	1\quad &\text{if } \# = a\\
		-1/2\quad &\text{if } j \in \{(0,1),(1,0),(-1,1),(2,0)\}, \# = b\\
		-1 \quad &\text{else}
	\end{cases}
\end{equ}
for values of $\alpha_1, \alpha_2 > 0$ to be chosen shortly and $\beta > 0$ a sufficiently small, free parameter. Throughout, we denote the set of indices for which the complex part is twice the real part in $q^*$ by $J_2 := \{(0,1),(1,0),(-1,1),(2,0)\}$.
Note that  $\eta^*\in T_{q^*}\mathcal Q$ since
\begin{equ}
	\eta^* \cdot \nabla E = \eta^*\cdot x^* = 0\qquad \text{and}\qquad \eta^* \cdot \nabla \mathcal E = \sum_{j \in \mathbb{Z}^2_N} \frac 1 {|j|^2}( (\eta^*)_j^a a_j^* + (\eta^*)_j^b b_j^*) = 0\,.
	\end{equ}
	Furthermore, it is easy to check that for every compatible pair $(E, \mathcal E)$, we have $\mathcal E(q^*) = \mathcal E$, $\|q^*\| = E$ upon choosing
	\begin{equ}
		\alpha_1  := \sqrt{\frac 1 {10} G_{\mathcal E, E} (\beta)} \,,\qquad
		\alpha_2  := \frac 1{\sqrt 2} \sqrt{E -2\beta^2 \pq{\frac {15} 8 + \sum_{k \in \mathbb{Z}^2_N\setminus J_2}\frac 1 {|k|^2}} - G_{\mathcal E, E} (\beta)}\,,
	\end{equ}
	where
	\begin{equ}
		G_{\mathcal E, E} (\beta) := \frac 1{1-\frac 1{2N^2}}\pc{\mathcal E - \frac 1 {2N^2}E - 2\beta^2 \pq{\pc{\frac {15} 8 - \frac 1 {N^2}} + \sum_{k \in \mathbb{Z}^2_N\setminus J_2}\pc{\frac 1 {|k|^2} - \frac 1 {2N^2}} }}
	\end{equ}
	for $\beta \in (0,Z_+(\mathcal E, E))$, $Z_+(\mathcal E, E)$ being the positive zero of $G_{\mathcal E, E} (\beta)$. Here, the expression $2\beta^2 \frac {15}8$ results from  $\frac {a_k^2}{|k^2|} +\frac {b_k^2}{|k^2|}$ for $k \in \{(-1,1),(0,2)\}$\,.

	We further note that, for any triple $(j,k,\ell) \in \{F(i)\}_{i \in [n-3]\setminus \{N+1\}}$, since $j \in J_2$, $k \not \in J_2 \cup\{ (N,N)\}$, the vector $(q^*, \eta^*)$ projected on the interacting coordinates can be written as
	\begin{equ}
(a_j,b_j,\eta_j^a,\eta_j^b,a_k,b_k,\eta_k^a, \eta_k^b,a_\ell,b_\ell, \eta_\ell^a, \eta_\ell^b) = (\gamma_1, 2\gamma_1, 1, -1/2 , \beta,\beta,,1,-1, \gamma_2, m \gamma_2,1,-1/m)\,,
	\end{equ}
with $\gamma_1 \in \{\alpha_1, \beta\}$, $\gamma_1 \in \{\alpha_1, \alpha_2, \beta\}$, $m = 2$ for $\ell \in \{ (0,1), (1,0),(-1,1),(0,2)\}$ and $m = 1$ otherwise.
	Evaluating $M_{jk\ell}'$ at $(q^*, \eta^*)$ for such interacting triples one therefore obtains a matrix of the form
\begin{equ}
	M_{jk\ell}' = \left(
\begin{array}{cccc}
 \beta  \gamma _1 C_{jk} & 0 & 2 C_{jk} C_{j\ell} \gamma _1^2 \gamma _2 m  & -\beta ^2 \gamma _2 m C_{jk} C_{k\ell} \\
 0 & \beta  \gamma _1 C_{jk} & -2 \gamma _1^2 \gamma _2 C_{jk} C_{j\ell} & \beta ^2 \gamma _2 C_{jk} C_{k\ell} \\
 \beta  C_{jk}+\gamma _1 C_{jk} & 0 & C_{jk} C_{j\ell}\pc{\frac{3}{2} \gamma _1 \gamma _2 m -\frac{2 \gamma _1^2 }{m}} & \frac{\beta ^2 C_{jk} C_{k\ell}}{m} \\
 0 & \beta  C_{jk}-\gamma _1 C_{jk} & C_{jk} C_{j\ell}\pc{-2 \gamma _1^2 -\frac{3}{2} \gamma _2 \gamma _1 } & \beta ^2 C_{jk} C_{k\ell} \\
\end{array}
\right)\,,
\end{equ}
whose determinant is given by
\begin{equ}\label{e:det1}
	\det (M_{jk\ell}') = \frac{3 \beta ^3 \gamma _1^3 \gamma _2 C_{jk}^4 C_{j\ell} C_{k\ell} \left(\beta +\beta  m^2+2 \gamma _2 m^2\right)}{2 m}\,.
	\end{equ}
	 For any choice of $\gamma_1 \in \{\alpha_1, \beta\}$, $\gamma_1 \in \{\alpha_1, \alpha_2, \beta\}$ and $m \in \{1,2\}$,  this determinant is a polynomial in $\beta \in (0,Z_+(\mathcal E, E))$ and as such is nonzero ouside on a set of measure $0$. This establishes that for any pair $(E, \mathcal E)$, $M_{F(i)}'$ for $i \in  \{4,\dots\}$ has full rank for almost every value of $\beta \in (0,Z_+(\mathcal E, E))$.

The only interactions that were not considered above are the ones corresponding to triples $F(N+1)$ and $F(0)$. A similar computation to the one carried out in the above paragraph, also carried out in \cite{code} and in Appendix~\ref{a:computations}, shows
	defining $M_{0}'$ as  $M_{0}$ without its first, second, eleventh and twelfth row, we have
\begin{equ}\label{e:det3}
	\text{det}(M_{0}') =-96 \alpha _1^5 \beta ^{11} C_{jk}^4 C_{j \ell}^5 C_{k\ell}^3\qquad \text{and} \qquad \text{det}(M_{F(N+1)}') =\frac{39}{2} \alpha _1^3 \beta ^4 \left(\beta -\alpha _1\right) C_{jk}^4 C_{j\ell} C_{k\ell}\,.
\end{equ}
	Again, these determinants are analytic functions of $\beta \in (0,Z_+(\mathcal E, E))$ and as such are nonzero except on a set of measure $0$. Combining \eref{e:det1} and \eref{e:det3} we see that the matrix $\mathcal A$ has rank $8 + 4\cdot(n-3) = 4n-4 = \text{dim}(T \mathcal X)$. Since upon changing $\alpha_1, \alpha_2$ this holds for any compatible choice  of the conserved quantities $(E, \mathcal E)$ we have shown the desired result.\qed\\[10pt]

\noindent\textbf{Acknowledgments:} \aa{We thank the anonymous referee
  for their thorough review and for pointing out a gap in the initial proof of \Cref{prop:dense}.} All authors thank the NSF  grant
NSF-DMS-1613337 for partial support during this project. AA also acknowledges the partial support of NSF-CCF-1934964 and GNAMPA - INdAM and OM thanks
NSF-DMS-2038056 for partial support during this project. JCM thanks David
Herzog  and Brendan Williamson for discussions at the start of these
investigations. JCM thanks the hospitality and support of the Institute
for Advanced Study. AA thanks the hospitality and the support of the Max Planck Institute of Mathematics in the Sciences and of the ScaDS AI institute of the University of Leipzig.


\bibliographystyle{plain}
\bibliography{refs}



\appendix

\section{Spanning computations for split Galerkin Euler equation}\label{a:computations}
The following Mathematica commands reproduce some of
the more tedious calculations from Section~\ref{sec:euler_alt2}. \om{There are three calculations: the determinants of two $4$-by-$4$ and one $8$-by-$8$ matrix, all of which can be computed by hand. In particular, these computer-based calculations are meant as an aid to the reader and are not a necessary component of our proof.} We have reformatted some of the outputs in \LaTeX\ to improve readability. The source Mathematica file can be found in \cite{code}.

\begin{mmaCell}[moredefined={V1, V2, V3, V4, W1, W2, W3, W4, V12, V13, V14, V23, V24, V34, W12, W13, W14, W23, W24, W34}]{Input}
V1=\{\mmaSub{C}{kl} \mmaSub{a}{k} \mmaSub{a}{l},0,\mmaSub{C}{jl} \mmaSub{a}{j} \mmaSub{a}{l},0,\mmaSub{C}{jk} \mmaSub{a}{j} \mmaSub{a}{k},0\};
V2=\{\mmaSub{C}{kl} \mmaSub{b}{k} \mmaSub{b}{l},0,0,\mmaSub{C}{jl} \mmaSub{a}{j} \mmaSub{b}{l},0,\mmaSub{C}{jk} \mmaSub{a}{j} \mmaSub{b}{k}\};
V3=\{0,\mmaSub{C}{kl} \mmaSub{a}{k} \mmaSub{b}{l},\mmaSub{C}{jl} \mmaSub{b}{j} \mmaSub{b}{l},0,0,\mmaSub{C}{jk} \mmaSub{b}{j} \mmaSub{a}{k}\};
V4=\{0,\mmaSub{C}{kl} \mmaSub{b}{k} \mmaSub{a}{l},0,\mmaSub{C}{jl} \mmaSub{b}{j} \mmaSub{a}{l},\mmaSub{C}{jk} \mmaSub{b}{j} \mmaSub{b}{k},0\};

W1=\mmaSub{\(\pmb{\partial}\)}{\{\{\mmaSub{a}{j},\mmaSub{b}{j},\mmaSub{a}{k},\mmaSub{b}{k},\mmaSub{a}{l},\mmaSub{b}{l}\}\}}V1.\{\mmaSubSup{\mmaUnd{\(\pmb{\eta}\)}}{j}{a},\mmaSubSup{\mmaUnd{\(\pmb{\eta}\)}}{j}{b},\mmaSubSup{\mmaUnd{\(\pmb{\eta}\)}}{k}{a},\mmaSubSup{\mmaUnd{\(\pmb{\eta}\)}}{k}{b},\mmaSubSup{\mmaUnd{\(\pmb{\eta}\)}}{l}{a},\mmaSubSup{\mmaUnd{\(\pmb{\eta}\)}}{l}{b}\}
W2=\mmaSub{\(\pmb{\partial}\)}{\{\{\mmaSub{a}{j},\mmaSub{b}{j},\mmaSub{a}{k},\mmaSub{b}{k},\mmaSub{a}{l},\mmaSub{b}{l}\}\}}V2.\{\mmaSubSup{\mmaUnd{\(\pmb{\eta}\)}}{j}{a},\mmaSubSup{\mmaUnd{\(\pmb{\eta}\)}}{j}{b},\mmaSubSup{\mmaUnd{\(\pmb{\eta}\)}}{k}{a},\mmaSubSup{\mmaUnd{\(\pmb{\eta}\)}}{k}{b},\mmaSubSup{\mmaUnd{\(\pmb{\eta}\)}}{l}{a},\mmaSubSup{\mmaUnd{\(\pmb{\eta}\)}}{l}{b}\}
W3=\mmaSub{\(\pmb{\partial}\)}{\{\{\mmaSub{a}{j},\mmaSub{b}{j},\mmaSub{a}{k},\mmaSub{b}{k},\mmaSub{a}{l},\mmaSub{b}{l}\}\}}V3.\{\mmaSubSup{\mmaUnd{\(\pmb{\eta}\)}}{j}{a},\mmaSubSup{\mmaUnd{\(\pmb{\eta}\)}}{j}{b},\mmaSubSup{\mmaUnd{\(\pmb{\eta}\)}}{k}{a},\mmaSubSup{\mmaUnd{\(\pmb{\eta}\)}}{k}{b},\mmaSubSup{\mmaUnd{\(\pmb{\eta}\)}}{l}{a},\mmaSubSup{\mmaUnd{\(\pmb{\eta}\)}}{l}{b}\}
W4=\mmaSub{\(\pmb{\partial}\)}{\{\{\mmaSub{a}{j},\mmaSub{b}{j},\mmaSub{a}{k},\mmaSub{b}{k},\mmaSub{a}{l},\mmaSub{b}{l}\}\}}V4.\{\mmaSubSup{\mmaUnd{\(\pmb{\eta}\)}}{j}{a},\mmaSubSup{\mmaUnd{\(\pmb{\eta}\)}}{j}{b},\mmaSubSup{\mmaUnd{\(\pmb{\eta}\)}}{k}{a},\mmaSubSup{\mmaUnd{\(\pmb{\eta}\)}}{k}{b},\mmaSubSup{\mmaUnd{\(\pmb{\eta}\)}}{l}{a},\mmaSubSup{\mmaUnd{\(\pmb{\eta}\)}}{l}{b}\};

V12=\mmaSub{\(\pmb{\partial}\)}{\{\{\mmaSub{a}{j},\mmaSub{b}{j},\mmaSub{a}{k},\mmaSub{b}{k},\mmaSub{a}{l},\mmaSub{b}{l}\}\}}V1.V2-\mmaSub{\(\pmb{\partial}\)}{\{\{\mmaSub{a}{j},\mmaSub{b}{j},\mmaSub{a}{k},\mmaSub{b}{k},\mmaSub{a}{l},\mmaSub{b}{l}\}\}}V2.V1;
V13=\mmaSub{\(\pmb{\partial}\)}{\{\{\mmaSub{a}{j},\mmaSub{b}{j},\mmaSub{a}{k},\mmaSub{b}{k},\mmaSub{a}{l},\mmaSub{b}{l}\}\}}V1.V3-\mmaSub{\(\pmb{\partial}\)}{\{\{\mmaSub{a}{j},\mmaSub{b}{j},\mmaSub{a}{k},\mmaSub{b}{k},\mmaSub{a}{l},\mmaSub{b}{l}\}\}}V3.V1;
V14=\mmaSub{\(\pmb{\partial}\)}{\{\{\mmaSub{a}{j},\mmaSub{b}{j},\mmaSub{a}{k},\mmaSub{b}{k},\mmaSub{a}{l},\mmaSub{b}{l}\}\}}V1.V4-\mmaSub{\(\pmb{\partial}\)}{\{\{\mmaSub{a}{j},\mmaSub{b}{j},\mmaSub{a}{k},\mmaSub{b}{k},\mmaSub{a}{l},\mmaSub{b}{l}\}\}}V4.V1;
V23=\mmaSub{\(\pmb{\partial}\)}{\{\{\mmaSub{a}{j},\mmaSub{b}{j},\mmaSub{a}{k},\mmaSub{b}{k},\mmaSub{a}{l},\mmaSub{b}{l}\}\}}V2.V3-\mmaSub{\(\pmb{\partial}\)}{\{\{\mmaSub{a}{j},\mmaSub{b}{j},\mmaSub{a}{k},\mmaSub{b}{k},\mmaSub{a}{l},\mmaSub{b}{l}\}\}}V3.V2;
V24=\mmaSub{\(\pmb{\partial}\)}{\{\{\mmaSub{a}{j},\mmaSub{b}{j},\mmaSub{a}{k},\mmaSub{b}{k},\mmaSub{a}{l},\mmaSub{b}{l}\}\}}V2.V4-\mmaSub{\(\pmb{\partial}\)}{\{\{\mmaSub{a}{j},\mmaSub{b}{j},\mmaSub{a}{k},\mmaSub{b}{k},\mmaSub{a}{l},\mmaSub{b}{l}\}\}}V4.V2;
V34=\mmaSub{\(\pmb{\partial}\)}{\{\{\mmaSub{a}{j},\mmaSub{b}{j},\mmaSub{a}{k},\mmaSub{b}{k},\mmaSub{a}{l},\mmaSub{b}{l}\}\}}V3.V4-\mmaSub{\(\pmb{\partial}\)}{\{\{\mmaSub{a}{j},\mmaSub{b}{j},\mmaSub{a}{k},\mmaSub{b}{k},\mmaSub{a}{l},\mmaSub{b}{l}\}\}}V4.V3;

W12=\mmaSub{\(\pmb{\partial}\)}{\{\{\mmaSub{a}{j},\mmaSub{b}{j},\mmaSub{a}{k},\mmaSub{b}{k},\mmaSub{a}{l},\mmaSub{b}{l}\}\}}V12.\{\mmaSubSup{\mmaUnd{\(\pmb{\eta}\)}}{j}{a},\mmaSubSup{\mmaUnd{\(\pmb{\eta}\)}}{j}{b},\mmaSubSup{\mmaUnd{\(\pmb{\eta}\)}}{k}{a},\mmaSubSup{\mmaUnd{\(\pmb{\eta}\)}}{k}{b},\mmaSubSup{\mmaUnd{\(\pmb{\eta}\)}}{l}{a},\mmaSubSup{\mmaUnd{\(\pmb{\eta}\)}}{l}{b}\};
W13=\mmaSub{\(\pmb{\partial}\)}{\{\{\mmaSub{a}{j},\mmaSub{b}{j},\mmaSub{a}{k},\mmaSub{b}{k},\mmaSub{a}{l},\mmaSub{b}{l}\}\}}V13.\{\mmaSubSup{\mmaUnd{\(\pmb{\eta}\)}}{j}{a},\mmaSubSup{\mmaUnd{\(\pmb{\eta}\)}}{j}{b},\mmaSubSup{\mmaUnd{\(\pmb{\eta}\)}}{k}{a},\mmaSubSup{\mmaUnd{\(\pmb{\eta}\)}}{k}{b},\mmaSubSup{\mmaUnd{\(\pmb{\eta}\)}}{l}{a},\mmaSubSup{\mmaUnd{\(\pmb{\eta}\)}}{l}{b}\};
W14=\mmaSub{\(\pmb{\partial}\)}{\{\{\mmaSub{a}{j},\mmaSub{b}{j},\mmaSub{a}{k},\mmaSub{b}{k},\mmaSub{a}{l},\mmaSub{b}{l}\}\}}V14.\{\mmaSubSup{\mmaUnd{\(\pmb{\eta}\)}}{j}{a},\mmaSubSup{\mmaUnd{\(\pmb{\eta}\)}}{j}{b},\mmaSubSup{\mmaUnd{\(\pmb{\eta}\)}}{k}{a},\mmaSubSup{\mmaUnd{\(\pmb{\eta}\)}}{k}{b},\mmaSubSup{\mmaUnd{\(\pmb{\eta}\)}}{l}{a},\mmaSubSup{\mmaUnd{\(\pmb{\eta}\)}}{l}{b}\};
W23=\mmaSub{\(\pmb{\partial}\)}{\{\{\mmaSub{a}{j},\mmaSub{b}{j},\mmaSub{a}{k},\mmaSub{b}{k},\mmaSub{a}{l},\mmaSub{b}{l}\}\}}V23.\{\mmaSubSup{\mmaUnd{\(\pmb{\eta}\)}}{j}{a},\mmaSubSup{\mmaUnd{\(\pmb{\eta}\)}}{j}{b},\mmaSubSup{\mmaUnd{\(\pmb{\eta}\)}}{k}{a},\mmaSubSup{\mmaUnd{\(\pmb{\eta}\)}}{k}{b},\mmaSubSup{\mmaUnd{\(\pmb{\eta}\)}}{l}{a},\mmaSubSup{\mmaUnd{\(\pmb{\eta}\)}}{l}{b}\};
W24=\mmaSub{\(\pmb{\partial}\)}{\{\{\mmaSub{a}{j},\mmaSub{b}{j},\mmaSub{a}{k},\mmaSub{b}{k},\mmaSub{a}{l},\mmaSub{b}{l}\}\}}V24.\{\mmaSubSup{\mmaUnd{\(\pmb{\eta}\)}}{j}{a},\mmaSubSup{\mmaUnd{\(\pmb{\eta}\)}}{j}{b},\mmaSubSup{\mmaUnd{\(\pmb{\eta}\)}}{k}{a},\mmaSubSup{\mmaUnd{\(\pmb{\eta}\)}}{k}{b},\mmaSubSup{\mmaUnd{\(\pmb{\eta}\)}}{l}{a},\mmaSubSup{\mmaUnd{\(\pmb{\eta}\)}}{l}{b}\};
W34=\mmaSub{\(\pmb{\partial}\)}{\{\{\mmaSub{a}{j},\mmaSub{b}{j},\mmaSub{a}{k},\mmaSub{b}{k},\mmaSub{a}{l},\mmaSub{b}{l}\}\}}V34.\{\mmaSubSup{\mmaUnd{\(\pmb{\eta}\)}}{j}{a},\mmaSubSup{\mmaUnd{\(\pmb{\eta}\)}}{j}{b},\mmaSubSup{\mmaUnd{\(\pmb{\eta}\)}}{k}{a},\mmaSubSup{\mmaUnd{\(\pmb{\eta}\)}}{k}{b},\mmaSubSup{\mmaUnd{\(\pmb{\eta}\)}}{l}{a},\mmaSubSup{\mmaUnd{\(\pmb{\eta}\)}}{l}{b}\};
\end{mmaCell}

\,\\
The computation for the determinant of the matrix $M_{F(i)}'$ for $i \in [n-3]\setminus \{N+1\}$ is given by
\begin{mmaCell}[moredefined={A1, V1, V2, V13, V34, W1, W2, W13, \
W34}]{Input}
A0 = Join[Transpose[\{V1,V2,V13,V34\}][[5;;]],Transpose[\{W1,W2,W13,W34\}][[5;;]]]
  /.\{\mmaSub{a}{j}->\mmaSub{\mmaUnd{\(\pmb{\alpha}\)}}{1}\
,\mmaSub{a}{k}->\mmaUnd{\(\pmb{\beta}\)},\mmaSub{a}{l}->\mmaSub{\mmaUnd{\(\pmb{\alpha}\)}}{2},\mmaSub{b}{j}->2*\mmaSub{\mmaUnd{\(\pmb{\alpha}\)}}{1},\mmaSub{b}{k}->\mmaUnd{\(\pmb{\beta}\)},\mmaSub{b}{l}->\mmaSub{\mmaUnd{\(\pmb{\alpha}\)}}{2},
      \mmaSup{\mmaSub{\mmaUnd{\(\pmb{\eta}\)}}{j}}{a}->1,\mmaSup{\mmaSub{\mmaUnd{\(\pmb{\eta}\)}}{k}}{a}->1,\mmaSup{\mmaSub{\mmaUnd{\(\pmb{\eta}\)}}{l}}{a}->1,\mmaSup{\mmaSub{\mmaUnd{\(\pmb{\eta}\)}}{j}}{b}->\
-1/2,\mmaSup{\mmaSub{\mmaUnd{\
\(\pmb{\eta}\)}}{k}}{b}->-1,\mmaSup{\mmaSub{\mmaUnd{\(\pmb{\eta}\)}}{l}}\
{b}->-1\};;
A0//MatrixForm
Simplify[Det[A0]]
\end{mmaCell}
\,

\noindent The matrix output is
\begin{equ}
	\qquad M_{F(i)}' = \left(
\begin{array}{cccc}
 \beta  \gamma _1 C_{jk} & 0 & \gamma _1^2 \gamma _2 j C_{jk} C_{j\ell} & -j\beta ^2 \gamma _2  C_{jk} C_{k\ell} \\
 0 & \beta  \gamma _1 C_{jk} & -\gamma _1^2 \gamma _2  C_{jk} C_{j\ell} & \beta ^2 \gamma _2 C_{jk} C_{k\ell} \\
 \beta  C_{jk}+\gamma _1 C_{jk} & 0 & C_{jk} C_{j\ell}\pc{-\frac{\gamma _1^2  }{j}+\gamma _2 \gamma _1 j -\gamma _2 \gamma _1 j }& \frac{\beta ^2 C_{jk} C_{k\ell}}{j} \\
 0 & \beta  C_{jk}-\gamma _1 C_{jk} & C_{jk} C_{j\ell} \pc{-\gamma _1^2  -\gamma _2 \gamma _1 k +\gamma _2 \gamma _1 } & \beta ^2 C_{jk} C_{k\ell} \\
\end{array}
\right)\,,
\end{equ}
and the corresponding determinant
\begin{equ}
	\det M_{F(i)}' = \frac{3}{2 j} C_{jk}^4 C_{jl} C_{kl} \beta ^3 \gamma _1^3 \gamma _2  \left(\beta +\beta  j^2+2 \gamma _2 j^2\right)\,.
\end{equ}\\[.3cm]
The computation for the determinant of the matrix $M_{F(N+1)}'$ is given by
\begin{mmaCell}[moredefined={A1, V1, V2, V13, V34, W1, W2, W13, \
W34}]{Input}
A1 = Join[Transpose[\{V1,V2,V13,V34\}][[5;;]],Transpose[\{W1,W2,W13,W34\}][[5;;]]]
  /.\{\mmaSub{a}{j}->\mmaSub{\mmaUnd{\(\pmb{\alpha}\)}}{1}\
,\mmaSub{a}{k}->\mmaUnd{\(\pmb{\beta}\)},\mmaSub{a}{l}->\mmaUnd{\(\pmb{\beta}\)},\mmaSub{b}{j}->2*\mmaSub{\mmaUnd{\(\pmb{\alpha}\)}}{1},\mmaSub{b}{k}->2*\mmaUnd{\(\pmb{\beta}\)},\mmaSub{b}{l}->\mmaUnd{\(\pmb{\beta}\)},
      \mmaSup{\mmaSub{\mmaUnd{\(\pmb{\eta}\)}}{j}}{a}->1,\mmaSup{\mmaSub{\mmaUnd{\(\pmb{\eta}\)}}{k}}{a}->1,\mmaSup{\mmaSub{\mmaUnd{\(\pmb{\eta}\)}}{l}}{a}->1,\mmaSup{\mmaSub{\mmaUnd{\(\pmb{\eta}\)}}{j}}{b}->-1/2,\mmaSup{\mmaSub{\mmaUnd{\
\(\pmb{\eta}\)}}{k}}{b}->-1/2,\mmaSup{\mmaSub{\mmaUnd{\(\pmb{\eta}\)}}{l}}{b}->-1\};;
A1//MatrixForm
Simplify[Det[A1]]
\end{mmaCell}
\,

\noindent The matrix output is
\begin{equ}
	M_{F(N+1)}' = \left(
\begin{array}{cccc}
 \alpha _1 \beta  C_{jk} & 0 & 2 \alpha _1^2 \beta  C_{jk} C_{j\ell} & -2 \beta ^3 C_{jk} C_{k\ell} \\
 0 & 2 \alpha _1 \beta  C_{jk} & -2 \alpha _1^2 \beta  C_{jk} C_{j\ell} & 2 \beta ^3 C_{jk} C_{k\ell} \\
 \alpha _1 C_{jk}+\beta  C_{jk} & 0 & \frac{3}{2} \alpha _1 \beta  C_{jk} C_{j\ell}-2 \alpha _1^2 C_{jk} C_{j\ell} & \frac{1}{2} \beta ^2 C_{jk} C_{k\ell} \\
 0 & 2 \beta  C_{jk}-\frac{\alpha _1 C_{jk}}{2} & -\frac{3}{2} \alpha _1 \beta  C_{jk} C_{j\ell}-2 \alpha _1^2 C_{jk} C_{j\ell} & \frac{7}{2} \beta ^2 C_{jk} C_{k\ell} \\
\end{array}
\right)\,,
\end{equ}
and the corresponding determinant
\begin{equ}
	\text{det} M_{F(N+1)}' = \frac{39}{2} \alpha _1^3 \beta ^4 \left(\beta -\alpha _1\right) C_{jk}^4 C_{j\ell} C_{k\ell}\,.
\end{equ}
\\[.3cm]

The computation for the determinant of the matrix $M_{0}'$ is given by
\begin{mmaCell}[addtoindex=4,moredefined={A2, V1, V2, V3, V4, V12, \
V13, V14, V23, W1, W2, W3, W4, W12, W13, W14, W23}]{Input}
A2 = Join[Transpose[\{V1,V2,V3,V4,V12,V13,V14,V23\}][[2;;5]],
  Transpose[\{W1,W2,W3,W4,W12,W13,W14,W23\}][[2;;5]]]
   /.\{\mmaSub{a}{j}\(\pmb{\to}\)\
\mmaSub{\mmaUnd{\(\pmb{\alpha}\)}}{1},\mmaSub{a}{k}\(\pmb{\to}\)\mmaUnd{\(\pmb{\beta}\)},\mmaSub{a}{l}\(\pmb{\to}\)\mmaUnd{\(\pmb{\beta}\)},\mmaSub{b}{j}\(\pmb{\to}\)2*\mmaSub{\mmaUnd{\(\pmb{\alpha}\)}}{1},\mmaSub{b}{k}\(\pmb{\to}\)\mmaUnd{\(\pmb{\beta}\)},\mmaSub{b}{l}\(\pmb{\to}\)\mmaUnd{\(\pmb{\beta}\)},
       \mmaSup{\mmaSub{\mmaUnd{\(\pmb{\eta}\)}}{j}}{a}\(\pmb{\to}\)1,\
\mmaSup{\mmaSub{\mmaUnd{\(\pmb{\eta}\)}}{k}}{a}\(\pmb{\to}\)1,\mmaSup{\mmaSub{\mmaUnd{\(\pmb{\eta}\)}}{l}}{a}\(\pmb{\to}\)1,\mmaSup{\mmaSub{\mmaUnd{\(\pmb{\eta}\)}\
}{j}}{b}\(\pmb{\to}\)-1/2,\mmaSup{\mmaSub{\mmaUnd{\(\pmb{\eta}\)}}{k}}{b}\(\pmb{\to}\)-1,\mmaSup{\mmaSub{\mmaUnd{\(\pmb{\eta}\)}}{l}}{b}\(\pmb{\to}\)-1\};
A2//MatrixForm
Simplify[Det[A2]]
\end{mmaCell}
\,

\noindent The first five columns of the matrix output are
\begin{equ}\small
	\begin{pmatrix}
 0 & 0 & \beta ^2 C_{k\ell} & \beta ^2 C_{k\ell} & 0  \\
 \alpha _1 \beta  C_{j\ell} & 0 & 2 \alpha _1 \beta  C_{j\ell} & 0 & \beta ^3 C_{j\ell} C_{k\ell}\\
 0 & \alpha _1 \beta  C_{j\ell} & 0 & 2 \alpha _1 \beta  C_{j\ell} & -\beta ^3 C_{j\ell} C_{k\ell} \\
 \alpha _1 \beta  C_{jk} & 0 & 0 & 2 \alpha _1 \beta  C_{jk} & \beta ^3 C_{jk} C_{k\ell}  \\
 0 & 0 & 0 & 0 & 0  \\
 \alpha _1 C_{j\ell}+\beta  C_{j\ell} & 0 & -2 \alpha _1 C_{j\ell}-\frac{\beta  C_{j\ell}}{2} & 0 & -\beta ^2 C_{j\ell} C_{k\ell} \\
 0 & \beta  C_{j\ell}-\alpha _1 C_{j\ell} & 0 & 2 \alpha _1 C_{j\ell}-\frac{\beta  C_{j\ell}}{2} & -\beta ^2 C_{j\ell} C_{k\ell} \\
 \alpha _1 C_{jk}+\beta  C_{jk} & 0 & 0 & -2 \alpha _1 C_{jk}-\frac{\beta  C_{jk}}{2} & -\beta ^2 C_{jk} C_{k\ell} \\
\end{pmatrix}\,,
\end{equ}
while the last three columns read
\begin{equ}\small
	\begin{pmatrix}
		 \alpha _1 \left(-\beta ^2\right) C_{j\ell} C_{k\ell} & \alpha _1 \left(-\beta ^2\right) C_{jk} C_{k\ell} & \alpha _1 \left(-\beta ^2\right) C_{jk} C_{k\ell}\\
		  0 & 2 \alpha _1^2 \beta  C_{jk} C_{j\ell} & -2 \alpha _1^2 \beta  C_{jk} C_{j\ell} \\
			0 & -2 \alpha _1^2 \beta  C_{jk} C_{j\ell} & 2 \alpha _1^2 \beta  C_{jk} C_{j\ell} \\
			 2 \alpha _1^2 \beta  C_{jk} C_{j\ell} & 0 & 0\\
			-\beta ^2 C_{j\ell} C_{k\ell} & -\beta ^2 C_{jk} C_{k\ell} & -\beta ^2 C_{jk} C_{k\ell} \\
		  0 & \frac{3}{2} \alpha _1 \beta  C_{jk} C_{j\ell}-2 \alpha _1^2 C_{jk} C_{j\ell} & 2 \alpha _1^2 C_{jk} C_{j\ell}-\frac{3}{2} \alpha _1 \beta  C_{jk} C_{j\ell}\\
			0 & -\frac{3}{2} \alpha _1 \beta  C_{jk} C_{j\ell}-2 \alpha _1^2 C_{jk} C_{j\ell} & \frac{3}{2} \alpha _1 \beta  C_{jk} C_{j\ell}+2 \alpha _1^2 C_{jk} C_{j\ell} \\
			 \frac{3}{2} \alpha _1 \beta  C_{jk} C_{j\ell}-2 \alpha _1^2 C_{jk} C_{j\ell} & 0 & 0
	\end{pmatrix}\,,
\end{equ}
and the corresponding determinant is
\begin{equ}
	\det M_0' = -96 \alpha _1^5 \beta ^{11} C_{jk}^4 C_{j\ell}^5 C_{k\ell}^3\,.
\end{equ}
%
%

\end{document}